\documentclass[a4paper]{article}
\usepackage[utf8x]{inputenc}
\usepackage{yfonts}
\usepackage{amsmath}
\usepackage{amsthm}
\usepackage{amssymb}
\usepackage{amsfonts}
\usepackage{paralist}
\usepackage[arrow, matrix, curve]{xy}
\usepackage{youngtab}
\usepackage{young}
\usepackage{graphicx}
\usepackage{wrapfig}
\usepackage{tikz}
\usepackage{mathtools}

\usetikzlibrary{calc}

\newtheorem{thm}[subsection]{Theorem}
\newtheorem{facts}[subsection]{Facts}

\newtheorem{lem}[subsection]{Lemma}

\newtheorem{cor}[subsection]{Corollary}

\newtheorem{prop}[subsection]{Proposition}
\newtheorem{construction}[subsection]{Construction}

\newtheorem{setup}[subsection]{Setup}

\newtheorem{notation}[subsection]{Notation}

\newtheorem{claim}[subsubsection]{Claim}

\theoremstyle{definition}
\newtheorem{rem}[subsection]{Remark}
\newtheorem{defi}[subsection]{Definition}
\newtheorem{ex}[subsection]{Example}

\newcommand{\A}{\mathbb{A}}
\newcommand{\V}{\mathbb{V}}

\newcommand{\F}{\mathbb{F}}

\newcommand{\C}{\mathbb{C}}
\newcommand{\E}{\mathbb{E}}

\renewcommand{\P}{\mathcal{P}}
\newcommand{\Q}{\mathbb{Q}}
\newcommand{\Z}{\mathbb{Z}}
\renewcommand{\H}{\mathrm{H}}
\renewcommand{\cal}{\mathcal}

\newcommand{\N}{\mathbb{N}}
\newcommand{\K}{\mathbb{K}}

\renewcommand{\O}{\mathbb{O}}

\renewcommand{\-}{\text{-}}

\renewcommand{\l}{\lambda}

\renewcommand{\v}{\textbf{v}}
\newcommand{\q}{\textbf{q}}

\newcommand{\rar}{\rightarrow}

\newcommand{\inj}{\hookrightarrow}

\newcommand{\surj}{\twoheadrightarrow}

\newcommand{\ol}{\overline}

\newcommand{\wt}{\widetilde}

\newcommand{\Irr}{\mathrm{Irr} \,}
\newcommand{\Proj}{\mathrm{Proj} \,}

\renewcommand{\mod}{\mathrm{mod}}
\newcommand{\grmod}{\mathrm{grmod}}

\newcommand{\Mod}{\mathrm{Mod}}

\newcommand{\per}{\mathrm{per}}
\newcommand{\D}{\mathrm{D}}

\newcommand{\IC}{\mathrm{IC}}

\newcommand{\Hom}{\mathrm{Hom}}
\newcommand{\RHom}{\mathrm{RHom}}
\newcommand{\Ext}{\mathrm{Ext}}
\newcommand{\End}{\mathrm{End}}
\newcommand{\grHom}{\mathrm{grHom}}
\newcommand{\grExt}{\mathrm{grExt}}
\newcommand{\grEnd}{\mathrm{grEnd}}

\newcommand{\real}{\mathrm{real}}

\newcommand{\BGS}{\mathrm{BGS}}
\newcommand{\spec}{\mathrm{spec}\,}
\renewcommand{\ker}{\mathrm{ker}}
\newcommand{\cok}{\mathrm{cok}}

\newcommand{\grdim}{\mathrm{\underline{dim }}\,}
\renewcommand{\dim}{\mathrm{dim }\,}

\newcommand{\Gr}{\mathrm{Gr}}

\newcommand{\Fr}{\mathrm{Fr}}

\newcommand{\SL}{\mathrm{SL}}

\begin{document}
 

\title{Grassmannians and Koszul duality}
\author{Jan Weidner}
\date{}
\maketitle

\begin{abstract}
Let $X$ be a partial flag variety, stratified by orbits of the Borel. We give a criterion for the 
category of modular perverse sheaves to be equivalent to modules over a Koszul ring.
This implies that modular category $\mathcal O$ is governed by a Koszul-algebra in small examples.
\end{abstract}

\section*{Introduction}

In the seminal paper \cite{BGS} Beilinson, Ginzburg and Soergel revealed deep Koszul duality phenomena in the representation theory of 
complex Lie algebras. They showed that certain categories of representations admit a natural grading, that was hidden before. 
More precisely let $B_{\C} \subset G_{\C}$ be a complex connected, simply connected reductive group, along with a Borel.
Then they constructed a Koszul ring whose finite dimensional modules are equivalent to $\mathcal O_0$ the principal block of category $\mathcal O$.

Where does this mysterious grading come from? The marvelous answer deserves to be recalled. 
The idea is, that the grading comes from ``weights'' alias Frobenius eigenvalues. 
Frobenius eigenvalues?! We are dealing with complex representations of a complex Lie-algebra... how on earth is this 
connected to the Frobenius homomorphism?

Well, the Beilinson Bernstein localization theorem combined with the Riemann Hilbert correspondence 
and a technical result of Soergel yields an equivalence of categories:
$$\mathcal O_0 \cong \cal P_{(B_{\C})}(X_{\C},\C)$$
The latter is the category of 
perverse sheaves on the complex flag variety $X_{\C}=G_{\C} /B_{\C}$ which are constant along $B_{\C}$-orbits.
Now reductive groups and flag varieties admit incarnations over any field. For example we could consider 
the cousin of our flag variety over the algebraic closure of a finite field $X_{\ol \F_q}=G_{\ol \F_q} /B_{\ol \F_q}$
and study perverse sheaves on it.
For technical reasons $\Q_l$ or $\ol \Q_l$ coefficients are preferable 
to complex coefficients in this setting. 
Anyway it turns out, that there is an equivalence of categories:
$$\cal P_{(B_{\C})}(X_{\C},\C) \cong \cal P_{(B_{\ol \F_q})}(X_{\ol \F_q},\ol \Q_l)$$
Now the situation looks much more Frobenius friendly than in the beginning.
And indeed a general philosophy asserts that things (varieties, sheaves,...) defined over $\F_q$ 
are equivalent to Frobenius equivariant things over $\ol \F_q$. 

For example if we want to produce a grading on the set of homomorphisms $\Hom(M,N)$ between 
two representations $M,N$ we proceed as follows: 
We consider the corresponding perverse sheaves $\cal M_{\ol \F_q},\cal N_{\ol \F_q}$ on $X_{\ol \F_q}$ and
investigate avatars $\cal M_{\F_q},\cal N_{\F_q}$ on $X_{\F_q}$ of them (``mixed sheaves''). 
These avatars are very non-unique and various choices of avatar pairs give rise to various Frobenius actions on 
$\Hom(M,N) \cong \Hom(\cal M_{\ol \F_q},\cal N_{\ol \F_q})$. 
However with a bit of luck and dexterity, we can arrange all Frobenius eigenvalues to be powers of $q$ and 
hence produce a grading on $\Hom(M,N)$. Applying this recipe to $M=N$ a projective generator of $\cal O_0$, 
one gets the Koszul grading of \cite{BGS}.

An important and widely open task is to extend the above to positive characteristics. 
From the geometric viewpoint this means, that we want to replace $\cal P_{(B_{\ol \F_q})}(X_{\ol \F_q},\Q_l)$
by
$$\cal P_{(B_{\ol \F_q})}(X_{\ol \F_q},\F_l)$$
But this screws up the grading! Over $\Q_l$ the possible eigenvalues $q^k , k\in \Z$ were all distinct, but 
now we have for example $q=q^l$.

Nevertheless a first breakthrough was recently obtained by Riche, Soergel and Williamson \cite{RSW}. In this article, we establish another more modest result. 
While the approach of \cite{RSW} is tailored towards the most complicated setup 
(full flag variety and no parity assumptions),
our strategy is designed with easier situations in mind (low dimension or special partial flag varieties). 
Benefits are better bounds on the characteristic and more flexibility on the input space, at the cost of a 
severe parity assumption. This flexibility makes it easier to adapt to the equivariant situation \cite{janEquivariant},
which was in fact the author's original motivation.

Also our line of argumentation is different from \cite{RSW}. While Riche, Soergel and Williamson essentially have to build 
up from scratch, the luxury of parity allows us to argue by comparison with characteristic zero.

Without further ado, let us come to the actual result.
Let $B_{\C}\subset P_{\C} \subset G_{\C}$ be a complex connected reductive group, along with a Borel and a parabolic subgroup. 
Let $X_{\C}=G_{\C}/P_{\C}$ be the corresponding partial flag variety, stratified by $B_{\C}$-orbits. 
Denote by $\cal P_{(B_{\C})}(X_{\C},\F_l)$ the category of perverse sheaves, which are constant along these orbits.
\begin{thm}\label{MainThmIntro}
Suppose that all $B_{\C}$-constructible $\mathbb Z_l$-intersection 
cohomology complexes $IC_w^{\Z_l}$ are parity sheaves. If $l>\mathrm wr(X)$\footnote{This is a mild and explicit condition on $l$ to be explained later \ref{defiWr}.}
then there exists a Koszul algebra $A$ and an equivalence of categories 
between perverse sheaves and finite dimensional modules over $A$:
$$\cal P_{(B_{\C})}(X_{\C},\F_l) \cong \mod \-A$$
\end{thm}
A concrete space where the above assumptions can be checked are the complex Grassmannians $\Gr(k,n)$ of $k$-planes inside $n$-space:
\begin{cor}
Let $X_{\C}:=\Gr(k,n)$ be a Grassmannian and $l>\min(k,n-k)+1$ be a prime. Then there exists a Koszul algebra $A$ and an equivalence of categories:
$$\cal P_{(B_{\C})}(X_{\C},\F_l) \cong \mod \-A$$
\end{cor}

\subsection*{Outline}
Let us give an outline of the proof of \ref{MainThmIntro}.
First of all, there are techniques to pass between objects (varieties, sheaves, etc) over $\C$ and their analogues over $\ol {\F}_q$.
The latter has the advantage, that one can bring the Frobenius into the picture, whose eigenspaces will ultimately 
yield the desired gradings. Hence we will work with varieties over $\ol \F_q$ 
and replace verbose notations $X_{\ol \F_q},\cal F_{\ol \F_q},\ldots$ by $X,\cal F,\ldots$ from now on.

Let $X=\bigsqcup_{\l \in \Lambda }X_\l$ be a nice cell stratified variety, for example a partial flag variety. 
For $\E=\Q_l,\Z_l,\F_l$ consider the category $\P_{\Lambda}(X,\E)$ of perverse sheaves. 
We construct projective covers $P_\l^\E\surj IC_\l^\E$, which can be assembled into a minimal projective generator $P^\E:=\bigoplus P_\l^\E$.
By abstract nonsense we have an equivalence
$$\P_\Lambda(X,{\E})\cong \mod \- \End(P^\E)$$
Now if $\E=\Q_l$ then the main statement \cite[4.4.4]{BGS} essentially says that $\End(P^{\Q_l})$ admits a Koszul grading coming from Frobenius eigenvalues.
Our goal is to show, that (under strong assumptions on $X$!) the analogue is
true in the modular situation. 
The general idea of the proof is to use the bridge
$$
\begin{xy}
 \xymatrix{ & \D_{\Lambda}^b(X,\Z_l) \ar[rd]^{\Q_l\otimes \_} \ar[ld]_{\F_l\otimes\-} & \\
\D_{\Lambda}^b(X,{\F_l}) & & \D_{\Lambda}^b(X,\Q_l)
}
\end{xy}
$$
between constructible derived categories, to pass results between modular characteristic and characteristic zero.

One always has $P^{\Z_l} \otimes \F_l=P^{\F_l}$ and under a (severe!) parity assumption, we deduce $P^{\Z_l} \otimes {\Q_l}=P^{\Q_l}$.
In this case we have the following relation between our endomorphism rings:
$$\End(P^{\F_l})={\F_l}\otimes \End(P^{\Z_l}) \text{ and } \End(P^{\Q_l})={\Q_l}\otimes \End(P^{\Z_l})$$
which we depict as 
$$
\begin{xy}
 \xymatrix{
\End(P^{\F_l}) & \ar@{~>}[l]_{{\F_l} \otimes } \End(P^{\Z_l}) \ar@{~>}[r]^{{\Q_l} \otimes } & \End(P^{\Q_l})\\
}
\end{xy}
$$
We already know that the right hand side admits a Koszul grading coming from Frobenius eigenvalues and we want to transfer it to the left. 
In order to do so, we will observe the following:
For a graded ${\Z_l}$-algebra $A$ satisfying some assumptions,
its base-change to $\F_l$ is Koszul if and only if its base-change to ${\Q_l}$ is. We will call such an algebra ${\Z_l}$-Koszul.

$$\text{$A\otimes {\F_l}$ is Koszul} \Leftrightarrow \text{$A$ is ${\Z_l}$-Koszul}  \Leftrightarrow \text{$A\otimes {\Q_l}$ is Koszul} $$ 


\subsection*{Acknowledgments}
I want to thank Wolfgang Soergel, Catharina Stroppel and Geordie Williamson for interest and helpful discussions.
Special thanks goes to an anonymous reviewer for spotting many typos and flaws in an older variant of this work.
This work was supported by the DFG via SPP 1388.


\section{The category of $\O$-perverse sheaves}
Let us introduce the categories of sheaves we work with. 

\subsubsection*{Basic notations}
Let $l \neq p$ be primes and $q$ be a power of $p$. 

Let $\K$ be a finite extension of $\Q_l$ and denote by $\O$ its ring of integers. Let $\varpi \in \O$ be a 
uniformizing parameter and $\F:=\O/\varpi$ be
the residue field. 
For instance $\K=\Q_l$ and $\O=\Z_l$ and $\F=\F_l$.

By $\E$ we denote any of $\K,\O,\F$.
The typical situation is that we are interested in $\F$-linear objects (representations, sheaves, $\ldots$) and 
we want to use $\O$ to connect them to their better understood counterparts over $\K$.

For a right (resp. left) noetherian ring $R$ we denote by $\mod \-R$ (resp. $R\-\mod$) the category 
of finitely generated right (resp. left) modules over $R$.

\subsubsection*{The six functors}
By $\D^b_c(X,\E)$ we denote the constructible derived category of a variety $X$ 
over a perfect field of characteristic different from $l$ with coefficients in $\E$.
We will refer to objects of $\D^b_c(X,\E)$ as constructible complexes.
These categories fit into a six functor formalism\footnote{From now on, we will often use 
the same notation for a functor and its derived counterpart. For example $\otimes$ means $\overset{L}{\otimes}$ etc.}
$$f_*,f^*,f_!,f^!,\cal Hom,\otimes$$
The six functors commute with extension of scalars and pullback from varieties $X_0$ over $\F_q$ to their basechange $X$ over $\ol \F_q$. 
For example we have a canonical isomorphism
$$\F \otimes \cal \Hom(\cal F, f_! \cal G)=\cal \Hom(\F\otimes \cal F,  f_! (\F\otimes\cal G))$$
where every functor is derived by our convention.
References for the six functor formalism are for example:
\begin{itemize}
\item \cite{SGA4} for $\E=\F$.
\item \cite{Ekedahl} for the passage from $\F$ to $\O$.
\item \cite{Weil2} for the passage from $\O$ to $\K$.
\end{itemize}

\subsection{Acyclically stratified varieties}
We adapt some basic definitions and notation from \cite{RSW}.
Let $X$ be a variety over a field $k$, together with a finite decomposition into locally closed smooth irreducible affine subvarieties
 $$X=\bigsqcup_{\l \in \Lambda} X_{\l}$$
such that the closure of each $X_{\l}$ is a union of some $X_{\mu}$.
We will denote the dimension of $X_\l$ by $d_\l$ and the inclusion by $l_\l:X_\l \inj X$. 
The inclusion of the closure of a stratum will be denoted by $\ol l_\l : \ol X_\l \inj X$.

If $k$ is algebraically closed, we say that $X=\bigsqcup_{\l \in \Lambda} X_\l$ is a stratification if
$$l^*_\l l_{\mu*} \E$$
has constant cohomology sheaves for all $\l,\mu$
\footnote{This definition is bad in general, usually one would only impose that the cohomology sheaves are local systems.
However for our purposes it is the most convenient. 
}. Here we also introduced the notation $\E$ for the constant sheaf 
on a space (in this case $X_\mu$).

A cell stratification is a stratification, such that
$X_\l \cong \A^{d_\l}$ for all $\l$. 
An acyclic stratification is a stratification, such that all strata are acyclic. Recall that a space $X_\l$ is called acyclic if
its cohomology ring is simply $\O$:
$$\H^\bullet(X_\l,\O)=\H^\bullet(pt,\O)=\O$$

Typical examples of cell stratified varieties are partial flag varieties $X:=G/P$ equipped with their decomposition into Bruhat cells:
$$G/P=\bigsqcup BwP/P$$
Examples of acyclically stratified varieties arise by taking suitable 
fiberbundles whose fibers and base are cell stratified.

Given an acyclically stratified variety $X=\bigsqcup X_\l$, we denote by $\D^b_\Lambda(X,\E)$ the category of all constructible complexes $\cal F$ such that 
$l_\l^*\cal F$ and $l_\l^! \cal F$ both have constant cohomology sheaves for all $\l$.
It is an idempotent complete triangulated category. 
We will be sloppy and usually refer to objects of $\D^b_\Lambda(X,\E)$ as sheaves. 
By $\cal P_\Lambda(X,\E)\subset \D^b_\Lambda(X,\E)$
we denote the full subcategory of perverse sheaves\footnote{In the case $\E=\O$ there are two dual categories, 
which one might call ``perverse sheaves''. We choose the one which is equivalent to the category of finitely generated $\O$-modules for $X=pt$. 
More precisely we use the $t$-structure $p_{1/2}$ and not $p_{1/2}^+$ in the terminology of \cite[3.3.4.]{BBD}.}.

From now on we work with objects (varieties, sheaves, \ldots) defined over a field $k$, where $k$ is either $\F_q$ or its algebraic closure $\ol \F_q$.
As usual objects over $\F_q$ are denoted by $X_0,\cal F_0,\ldots$, while their 
base-change to $\ol \F_q$ is denoted by $X,\cal F,\ldots$.

We say that a locally closed decomposition $X_0=\bigsqcup_{\l \in \Lambda} X_{\l,0}$ is a (acylcic/cell) stratification if its basechange
$X=\bigsqcup_{\l \in \Lambda} X_{\l}$ is. In this case
we denote by $\D^b_\Lambda(X_0,\E)$ all constructible complexes whose base-change lands in $\D^b_\Lambda(X,\E)$.
Again this is an idempotent complete triangulated category.
The category $\cal P_\Lambda(X_0,\E)\subset \D^b_\Lambda(X_0,\E)$ is defined in a similar way.

\begin{ex}
 Suppose that $X=X_\l=\mathbb A^n$ consists of a single cell. Then we have 
$$\D^b_\Lambda(X,\E) \cong \D^b(\mod\-\E) \text{ and } \cal P_\Lambda(X,\E) \cong \mod\-\E$$
The key point is that $\RHom(\E,\E)=\H^\bullet(X,\E)=\E$.
\end{ex}

\begin{notation}
 Let $X$ be an acyclically stratified variety, and $X_\l$ be a stratum. Then there are a couple of canonically associated perverse sheaves
on $X$. We will introduce notation for them here:
\begin{itemize}
\item Let $\Delta_\l:=\Delta^{\E}_\l:={l_\l}_! \E[d_\l]$ denote the standard perverse sheaf.
\item Let $\nabla_\l:=\nabla^{\E}_\l:={l_\l}_* \E[d_\l]$ denote the costandard perverse sheaf.
\item Let $\IC_\l:=\IC^{\E}_\l:={l_\l}_{!*} \E[d_\l]$ denote the intersection cohomology complex. 
\end{itemize}
If the stratification is defined over $\F_q$, the same formulas define $\Delta_{\l,0},\nabla_{\l,0},\IC_{\l,0}$.
\end{notation}

\subsection{Basic properties of sheaves on acyclically stratified varieties}
A good reference for perverse sheaves in general is \cite{BBD}. For $\O$-linear perverse sheaves 
on acyclically stratified varieties \cite{RSW} contains foundational information. 

Let $\cal T$ be a triangulated category equipped with a $t$-structure and heart $\cal A$. 
In this generality there is no natural functor $\D^b(\cal A)\rar \cal T$. 
However in every ``natural'' situation, for example in the case of the perverse $t$-structure,
there is a canonical realization functor 
$$\real: \D^b(\cal A)\rar \cal T$$
see \cite[3.1.9]{BBD}, \cite{DbPerv}.
\begin{thm}\label{thmRealEquiv}
 Let $X_0=\bigsqcup X_{\l,0}$ be an acyclically stratified variety. Then the realization functors 
 are equivalences of categories, commuting with the forgetful functor:
 $$
 \begin{xy}
  \xymatrix{
  \D^b(\P_\Lambda(X_0,\E)) \ar[d] \ar[r]^{\real} & \D^b_\Lambda(X_0,\E) \ar[d] \\
  \D^b(\P_\Lambda(X,\E)) \ar[r]^{\real} & \D^b_\Lambda(X,\E) 
  }
 \end{xy}
$$
 \begin{proof}
  \cite[3.4.1.]{RSW}, \cite[2.3.4]{RSW}
 \end{proof}
\end{thm}

\begin{prop}\label{thmFiniteCohoDim}
Let $X=\bigsqcup X_{\l}$ be an acyclically stratified variety. Then the following assertions hold:
\begin{enumerate}
\item The category $\P_\Lambda(X,\E)$ has enough projectives.
\item The category $\P_\Lambda(X,\E)$ has finite projective dimension.
\item The category $\P_\Lambda(X,\E)$ is noetherian. If $\E=\F,\K$ it is also artinian. 
\end{enumerate}
\begin{proof}
 \begin{enumerate}
\item \cite[2.3.3]{RSW}
\item \cite[2.3.4]{RSW}
\item \cite[4.3.1]{BBD} , \cite[4.0.(b)]{BBD}
\end{enumerate}
\end{proof}
\end{prop}

\subsubsection*{Perverse sheaves as modules}
It is often useful to know, that the category of perverse sheaves can be realized as modules over a ring:
\begin{thm}\label{thmMoritaProj}
Let $\cal A$ be a noetherian
abelian category and $P\in \cal A$ be a projective generator. Then $\End(P)$ is right noetherian and 
\begin{equation*}
\left\{
\begin{aligned}
\cal A &\to \mod \-\End(P) \\
M & \mapsto  \cal A(P,M)
\end{aligned}
\right.
\end{equation*}
defines an equivalence between $\cal A$ and the category of finitely generated right modules over $\End(P)$.

Furthermore if $P=\bigoplus P_\l$ is a decomposition, then projection onto $P_\l$ gives an idempotent $e_\l$ in $A:=\End(P)$
and under the above equivalence we have
$$P_\l \mapsto P_\l:=e_\l A$$
\begin{proof}
 \cite[page 55]{BassKtheo} 
\end{proof}
\end{thm}

\subsubsection*{Perverse sheaves as $(A,\phi)$-modules}

Theorem \ref{thmMoritaProj} enables us to realize $\P_\Lambda(X,\E)$ as a category of modules. However it does not apply to 
$\P_\Lambda(X_0,\E)$ since there are not enough projectives in this category. 
The solution is to interpret $\P_\Lambda(X_0,\E)$ as Frobenius equivariant objects in $\P_\Lambda(X,\E)$ 
and translate Frobenius equivariance into the language of modules. This strategy was carried out in \cite{RSW}.
Let us quickly recall what we need.
\begin{defi}
Let $A$ be an $\E$-algebra equipped with an automorphism $\phi$.
We denote by $\Mod \-(A,\phi)$ the category of pairs $(M,\phi)$ 
where $M$ is a right $A$-module and $\phi:M\rar M$ is an $\E$-linear bijection
satisfying $\phi(m\cdot a)=\phi(m) \cdot \phi(a)$.
A morphism $f:(M,\phi)\rar (N,\phi)$ is an $A$-linear map that commutes with $\phi$.
\end{defi}
If $A$ is right noetherian we denote by $\mod \-(A,\phi)\subset \Mod \-(A,\phi)$ the subcategory 
of objects $(M,\phi)$ for which $M$ is noetherian.
\begin{lem}
Let $\E=\F,\O$. Then the functor $\P_\Lambda(X_0,\E)\rar \P_\Lambda(X,\E)$ induces an equivalence
 between $\P_\Lambda(X_0,\E)$ and the category of pairs $(\cal F,F_{\cal F})$, where $\cal F \in \P_\Lambda(X,\E)$ and
 $F_{\cal F}:\Fr^* \cal F \rar \cal F$ is an isomorphism.
 
 Here $\Fr$ is the geometric Frobenius.
 \begin{proof}
  \cite[5.1.2]{BBD}
 \end{proof}
\end{lem}

\begin{ex}
 Let $\cal F,\cal G \in \P_\Lambda(X,\O)$ be two perverse sheaves equipped with lifts $\cal F_0,\cal G_0$ to $X_0$. 
 Then $\Hom(\cal F,\cal G)$ is equipped with a Frobenius automorphism constructed as follows:
 $$
 \begin{xy}
  \xymatrix{
  \Hom(\cal F,\cal G) \ar[r]^{\Fr^*} & \Hom(\Fr^* \cal F,\Fr^* \cal G) \ar[rr]^{F_{\cal G} \circ \_ \circ F_{\cal F}^{-1}} && \Hom(\cal F,\cal G)
  }
 \end{xy}
$$
In particular let $P\in \P_\Lambda(X,\O)$ be a projective generator with lift $P_0$. Then $\End(P)$ is equipped with a 
Frobenius automorphism.
\end{ex}

\begin{prop}
 Let $X_0=\bigsqcup X_{\l,0}$ be an acyclically stratified variety, $P\in \P_\Lambda(X,\O)$ be a projective generator 
 and $P_0$ be a lift. Let $(A,\phi)$ be $\End(P)$ equipped with the Frobenius action.
 
 Then $\Hom(P,\_)$ induces an equivalence:
 $$\P_\Lambda(X_0,\O) \rar \mod \-(A,\phi)$$
which commutes with the ``forgetful'' functors on both sides.
 \begin{proof}
  \cite[3.3.7]{RSW}
 \end{proof}
\end{prop}

\subsection{Parity sheaves}
The reference on parity sheaves is \cite{JMWparity}.

An object $\cal F\in \D^b_\Lambda(X,\E)$ is called 
even (odd) if for all $\l$ the objects $l^*_\l(\cal F)$ and $l^!_\l(\cal F)$ 
have constant torsion free cohomology sheaves, which vanish in odd (even) degrees.
An object $\cal F\in \D^b_\Lambda(X,\E)$ is called parity, if it is a direct sum of an even and an odd object.

Let $X=\bigsqcup X_\l$ be a stratified variety. We say that $X$ satisfies $\IC^\O$-parity, if for all $\l$ the object $\IC^\O_\l$ is parity.
\begin{lem}
We have $\K\otimes \IC^\O_\l=\IC_\l^\K$. If $\IC^\O_\l$ is parity, then we also have $\F \otimes \IC^\O_\l=\IC^\F_\l$ and the latter is also parity.
\begin{proof}
 Since the (co)stalks of parity sheaves are free over $\O$, we see that
$\F \otimes \IC^\O_\l$ satisfies the (co)stalk conditions for being an $\IC$-sheaf.
\end{proof}
\end{lem}

\begin{lem}\label{LemExtGrpsTheSame}
 Let $X$ be a cell stratified variety such that $\IC_\l^\O,\IC_\mu^\O$ are parity. Then for all $i$ the module
$$\Ext^i(\IC_\l^\O,\IC_\mu^\O)$$
is free over $\O$. We obtain in particular 
$$\Ext^i(\IC_\l^\F,\IC_\mu^\F)=\F \otimes \Ext^i(\IC_\l^\O,\IC_\mu^\O) \text{ and } \Ext^i(\IC_\l^\K,\IC_\mu^\K)=\K \otimes \Ext^i(\IC_\l^\O,\IC_\mu^\O)$$
\begin{proof}
\cite[2.6]{JMWparity}
\end{proof}
\end{lem}

%
%
%
%
%
%

\section{$\O$-Koszul algebras}
Recall that a non-positively\footnote{Actually most authors consider non-negatively graded rings $A=\bigoplus A^i$
with $\Ext^i$ in degree $i$.} graded ring $A=\bigoplus A_i$ such that 
$A_0$ is semisimple is called Koszul if for all $i\in \N$ the graded module $\Ext^i(A_0,A_0)$ is concentrated in degree $-i$.

Suppose now that $A^\F=A\otimes \F$ is an $\F$-algebra, which admits an ``integral form'' $A$ over $\O$
such that $A\otimes \K$ is Koszul.
Is it true that $A^\F$ is Koszul as well? The answer is positive, if $A$ satisfies some reasonable assumptions.

\begin{setup}\label{setupGraded}
 Let $A$ be a right noetherian non-positively graded $\E$-algebra, which is free as an $\E$-module. Suppose that
$$A_0=\prod_{\l \in \Lambda} \E$$
as an $\E$-algebra for some finite index set $\Lambda$.
Denote by $e_{\l}$ the idempotent corresponding to the $\l$-th copy of $\E$.
Let us define the following graded right modules:
$$
\begin{xy}
 \xymatrix{P_\l:=e_\l A & P:=\underset{\l \in \Lambda}{\bigoplus} P_\l \\
L_\l:=P_\l / P_\l^{<0}  & L:=\underset{\l \in \Lambda}{\bigoplus} L_\l
}
\end{xy}
$$
Here $M^{<0}:=\bigoplus_{i<0} M_i$ for a graded module $M$.
We think of $L_\l$ as simple modules and $P_\l$ as their projective covers.
\end{setup}

\begin{rem}
 If $\E=\F,\K$ is a field, then the $L_\l$ are precisely the simple graded modules. 
 If $\E=\O$, then the $L_\l$ are not simple objects in the abelian category of graded $A$-modules.
 
 However they are still simple in less naive sense: 
 Let $\cal A$ be an exact category. We say that an object $M\in \cal A$ is simple if 
 it is non-zero and its only admissible quotients are $M$ and $0$.
 Now the $L_\l$ are indeed the simple objects in the exact category of graded $A$-modules, which are free over $\O$.
\end{rem}

For two graded modules $M,N$ over a graded ring there is a grading on $\Hom(M,N)$,
the space of module homomorphisms. We will call it ``internal grading'', 
as well as its cousin on $\Ext^i(M,N)$.

Given a (graded) ring $R$, we denote by $R\-\mathrm{(gr)Mod}$ its category of (graded) left modules.
We also make use of variants of this notation. For instance if $R$ is a graded right noetherian ring,
we denote by $\grmod\-R$ the category of finitely generated graded right modules over $R$.

\begin{lem}\label{lemExtensionOfScalars}
 Let $A$ be a graded right noetherian $\O$-algebra, which is free as an $\O$-module. Let $M,N \in \D^b(\grmod\-A)$
be two objects of the bounded derived category. Then there are canonical isomorphisms:
\begin{itemize}
\item $\F \overset{L}{\otimes} \RHom_A(M,N) = \RHom_{A\otimes \F}(\F\overset{L}{\otimes} M, \F\overset{L}{\otimes} N)$
\item $\K \overset{L}{\otimes} \RHom_A(M,N) = \RHom_{A\otimes \K}(\K\overset{L}{\otimes} M, \K\overset{L}{\otimes} N)$
\end{itemize}
\begin{proof}
 Let us focus on the $\F$-case. Observe that for $M,N$ graded free of finite rank, on the one hand 
$\RHom_A(M,N)$ is free over $\O$. On the other hand $\F \overset{L}{\otimes} M$ and $\F\overset{L}{\otimes} N$ 
are graded free modules again.
This shows that the functors on both sides can be computed in the naive way on complexes of graded free modules.
Hence one can  write down a natural isomorphism $\F \overset{L}{\otimes} \RHom_A(M,N) \rightarrow \RHom_{A\otimes \F}(\F\overset{L}{\otimes} M, \F\overset{L}{\otimes} N)$
for such complexes in the naive way. 

\end{proof}
\end{lem}
Let us consider the ext-algebra $E:=\Ext^\bullet(L,L)$. It is equipped with two gradings, the cohomological and the internal one.

\begin{defi}
 Let $A$ be as in \ref{setupGraded}. Assume that $E$ is free as an $\E$-module.
We say that $A$ is $\E$-Koszul, if the internal grading on $E$ is the 
negative of the cohomological grading on $E$. In other words, we demand that $\Ext^i(A_0,A_0)$ is concentrated in internal degree $-i$.
\end{defi}
If $\E$ is a field, than $A$ is $\E$-Koszul, if and only if it is Koszul in the usual sense (after inverting the sign of the grading).

\begin{thm}\label{thmOKoszul}
Let $A$ be an $\O$-algebra as in \ref{setupGraded} and suppose that $\Ext^\bullet(L,L)$ is free over $\O$. Then the following are equivalent:
\begin{itemize}
\item $A$ is $\O$-Koszul.
\item $A\otimes \F$ is $\F$-Koszul.
\item $A\otimes \K$ is $\K$-Koszul.
\end{itemize}
\begin{proof}
By \ref{lemExtensionOfScalars} we have $\Ext^i(\F\otimes L,\F\otimes L)=\F\otimes \Ext^i(L,L)$ and $\Ext^i(\K\otimes L,\K\otimes L)=\K\otimes \Ext^i(L,L)$.
Now it remains to observe that for a free graded $\O$-module the property of being concentrated in degree $-i$
is reflected and conserved by tensoring with $\F,\K$.
\end{proof}
\end{thm}


\begin{rem}
 Let $A$ be $\E$-Koszul, say of finite cohomological dimension. Using standard techniques \cite[Chapter 2]{OlafThesis}, one can show that 
 there is an equivalence of categories 
 $$\mod \-A \cong \per \-\Ext^\bullet(L)$$
 Here we consider $\Ext^\bullet(L):=(\Ext^\bullet(L,L),0 )$ as a dg-algebra with trivial differential and 
 $\per \-\Ext^\bullet(L)$ denotes the perfect derived category. In other words $\per \-\Ext^\bullet(L)$ 
 is the smallest full idempotent complete triangulated subcategory of the derived category of dg-modules over $\Ext^\bullet(L)$ which contains
 $\Ext^\bullet(L)$.
\end{rem}

\section{Projectives in $\P_{\Lambda}(X,\E)$}

\subsection{Multiplicities}
We start with some general remarks on multiplicities in abelian categories.

Let $\cal A$ be an abelian category linear over a field $k$. Then we are interested in its Grothendieck group $K(\cal A)$.
For any basis $(L_i)$ of $K(\cal A)$ and an element $M\in K(\cal A)$ we define the multiplicity of 
$L_i$ in $M$ by the formula 
$$M=\sum [M:L_i] \cdot L_i$$
Assume that $\cal A$ is also of finite cohomological dimension and with finite dimensional
morphism spaces. Then we have a pairing $K(\cal A) \times K(\cal A) \rightarrow \Z$
induced by taking the Euler characteristic of the $\Hom$-complex:
$$h([x],[y]):=\chi_{\RHom(x,y)}:= \sum (-1)^i \dim  \Ext^i(x,y) \text{ for } x,y\in \cal A$$

In this setting we have BGG-reciprocity:
\begin{lem}\label{lemBGG}
Let $(P_i),(L_j)$ and $(\Delta_i),(\nabla_j)$ be two collections of dual bases of $K(\cal A)$, in the sense that
$$h(P_i,L_j)=h(\Delta_i,\nabla_j)=\delta_{ij}$$
Then we have 
$$[P_i:\Delta_j]=[\nabla_j:L_i]$$
\begin{proof}
For any $M,N$ we have $[M:L_i]=h(P_i,M)$ and $h(M,\nabla_j)=[M:\Delta_j]$.
Hence we may compute:
$$[\nabla_j:L_i]=h(P_i,\nabla_j)=[P_i:\Delta_j]$$
\end{proof}
\end{lem}

\subsection{Projectives in $\P_{\Lambda}(X,\E)$}
In \cite{RSW} it was shown, that there always exist indecomposable projectives and maps $P_\l^\O\surj \IC_\l^\O$ in
the category $\P_{\Lambda}(X,\O)$. However in this generality $P_\l^\O$ is not very explicit. 
For example there is no description of $\Hom(P_\l^\O,\IC_\mu^\O)$! The situation is much more transparent if 
the $\IC$-sheaves are parity:
\begin{prop}\label{ThmProjX}
Let $X$ be an acyclically stratified variety which satisfies $\IC^\O$-parity.
Let $X_\l$ be any stratum.
\begin{enumerate}
\item Then there exists a projective object $P_\l^\E \in \P_{\Lambda}(X,\E)$ such that 
$$\Hom(P_\l^\E,\IC_\mu^\E)\cong\E^{\delta_{\l,\mu}}$$
In particular $P_\l^\E$ is indecomposable.
\item We have 
$$P_\l^\F = \F \otimes P_\l^\O \text{ and } P_\l^\K = \K \otimes P_\l^\O  $$
\end{enumerate}
\end{prop}
 The explicit construction of the projectives in \ref{ThmProjX} will be important. It is a variant of \cite[2.3.1]{RSW}.
We will give it in a moment and justify it (e.g. show that it indeed gives projectives etc.) afterwards.
\begin{construction}\label{constrproj}
The construction goes by recursion on the number of strata of $X$. If $X= \emptyset$, there is nothing to do.
Now assume that $X$ consists of $n$ strata and we already constructed projectives for smaller varieties.
\begin{itemize}
 \item If $X_\l$ is an open stratum, we put $P_\l^\E:=\Delta^\E_\l$ and see by adjunction that $P_\l^\E$ is projective: 
Indeed $\Hom(P_\l^\E,\_)=\Hom(\E[d_\l],l_\l^! \_)$ is exact since $l_\l$ is an open inclusion.
This description also shows $\Hom(P_\l^\E,\IC_\mu^\E)=\E^{\delta_{\l,\mu}}$.

\item If $X_\l$ is not open, let $U:=X_\mu$ be an open stratum and $Y$ be its complement.
By assumption we have already constructed a projective $P_{Y,\l}^\E$ with the desired properties on $Y$. Let 
$$E:=E^\E:=\Ext^1(P_{Y,\l}^\E,\Delta^\E_\mu)$$
\end{itemize}
\begin{claim}\label{ClaimEoFree}
 $E^\E$ is free over $\E$.
\end{claim}
The claim is obvious for $\E=\F,\K$ and we will see later that it holds also for $\E=\O$.
Hence it makes sense to consider the (naive) dual $E^*:=\Hom(E,\E)$ of $E$. The identity
element in $E^*\otimes E=\Ext^1(P_{Y,\l}^\E,E^*\otimes\Delta^\E_\mu)$ gives rise to a canonical extension, which defines $P_{\l}^\E$:
\begin{equation}\label{EqnProjDefSeq}
 0 \rightarrow E^*\otimes \Delta_\mu \rightarrow P_{\l}^\E \rightarrow P_{Y,\l}^\E \rightarrow 0
\end{equation}

\begin{proof}[Proof of \ref{ClaimEoFree}]
We want to show that $E^\O$ is also free.

By structure theory of modules over a principal ideal domain, we know that $E$ is of the form
$$E^\O\cong \O^n \oplus \bigoplus_{i=1}^d \O/\varpi^{n_i}$$
Since extension of scalars commutes with $\RHom$, we have $\dim  E^\K=n$ and $\dim  E^\F \geq n+d$. 
Here we used that any complex of modules over a principal ideal domain is quasi-isomorphic to its cohomology.
If we can show $\dim  E^\F=\dim  E^\K$ we are done. 

Copying the proof of \cite[3.2.1]{BGS}, we know that 
$$\Hom(P_\l^\K,\IC_\mu^\K)=\K^{\delta_{\l,\mu}} \text{ and } \Hom(P_\l^\F,\IC_\mu^\F)=\F^{\delta_{\l,\mu}}$$ 
Now using \ref{lemBGG},\ref{lemParityGoodMult} we compute:

\begin{align*}
 \dim  E^\F &= [P^\F_\l:\Delta^\F_\mu]\\
&=[\nabla^\F_\mu : \IC^\F_\l]\\
&=[\nabla^\K_\mu : \IC^\K_\l]\\
&=[P^\K_\l:\Delta^\K_\mu] \\
&=\dim  E^\K
\end{align*}
\end{proof}
\begin{claim}
 $P_{\l}^\E$ is projective
\begin{proof}
This can be shown as in \cite[2.3.1]{RSW}.
\end{proof}
\end{claim}
We now have our projectives and freeness of $E$ implies 
$$P_\l^\F = \F \otimes P_\l^\O \text{ and } P_\l^\K = \K \otimes P_\l^\O  $$
It remains to check the ``orthonormality`` relation between projectives and $IC$-sheaves:
\begin{claim}
$\Hom(P_\l^\E,\IC_\mu^\E)=\E^{\delta_{\l,\mu}}$
\begin{proof}
If $\E=\K,\F$, this can be shown as in the proof of \cite[3.2.1]{BGS}. 
Since 
$$\F^{\delta_{\l,\mu}}=\RHom(P_\l^\F,\IC_\mu^\F)=\F \overset{L}{\otimes} \RHom(P_\l^\O,\IC_\mu^\O)$$
it holds also for $\E=\O$. Here we used parity.
\end{proof}
\end{claim}
\end{construction}

\begin{lem}\label{lemParityGoodMult}
 Let $X$ be an acyclically stratified variety, which satisfies $\IC^\O$-parity. Then for all $\l,\mu$
the multiplicities 
$$[\Delta_\l^\F:\IC_\mu^\F]=[\Delta_\l^\K : \IC_\mu^\K]$$
 coincide.
\begin{proof}
For the duration of this proof let $\E$ be a field (i.e. $\E=\K$ or $\E=\F$). 
Both the $\IC$-sheaves and the standard sheaves form a basis of the Grothendieck group of 
$\cal P_\Lambda(X,\E)$. We only need to check that the ``inverse multiplicities''
$$[\IC_\mu^\F:\Delta_\l^\F]\overset{!}{=}[\IC_\mu^\K:\Delta^\K_\l]$$
coincide. Using $h(\Delta_\l,\nabla_\mu)=\delta_{\l,\mu}$ we compute:
\begin{align*}
 [\IC_\mu^\E:\Delta_\l^\E]&=h(\IC_\mu^\E,\nabla_\l^\E) \\
&=\chi_{\RHom(\IC_\mu^\E,\nabla_\l^\E)} \\
&=\chi_{\RHom(l_\l^* \IC_\mu^\E,\E[d_\l])} 
\end{align*}
Hence it suffices to show that 
$$\dim  \Ext^i(l_\l^* \IC_\mu^\F,\F[d_\l])\overset{!}{=}\dim  \Ext^i(l_\l^* \IC_\mu^\K,\K[d_\l])$$
By parity we know that $\Ext^i(l_\l^* \IC_\mu^\O,\O[d_\l])$ is free (\ref{LemExtGrpsTheSame}) and $\E\otimes \IC^\O=\IC^\E$.
Since the six functors commute with extension of scalars we are done.
\end{proof}
\end{lem}

\begin{rem}\label{remGenerator}
 Note for later use, that in the situation of \ref{ThmProjX} the map $P_\l\surj \Delta_\l \surj \IC_\l$ generates $\Hom(P_\l,\IC_\l)$.
We have just seen, that this space is always free of rank $1$. Now if $\E$ is a field, the assertion follows since the map is non-zero. 
The case $\E=\O$ is a consequence of the case $\E=\F$ by a Nakayama style argument.
\end{rem}

\begin{rem}
In \cite[2.4.2.]{RSW} uniqueness properties of projective perverse sheaves are discussed. It follows 
that the $P_\l^\E$ are the only indecomposable projective objects in $\P_{\Lambda}(X,\E)$. 

With additional effort, one can even show that $P_\l^\E\rar \IC_\l^\E$ is a projective cover in $\P_{\Lambda}(X,\E)$.
\end{rem}

\begin{lem}\label{lemHomProjFree}
 Let $X$ be an acyclically stratified variety. Then for all $\l,\mu$ 
 $$\Hom(P_\l^\O,P_\mu^\O)$$
 is a free $\O$-module.
 \begin{proof}
 This is proven in \cite[2.4.1]{RSW}. Under the assumption of $\IC_\l^\O$-parity
one can alternatively exploit coincidence of multiplicities and deduce
$$\dim \K \otimes \Hom(P_\l^\O,P_\mu^\O) = \dim \F \otimes \Hom(P_\l^\O,P_\mu^\O)$$
\end{proof}
\end{lem}

\section{Graded categories}
We recall some notions and results from \cite{BGS}.
\begin{defi}
Let $\wt {\cal C}, \cal C$ be abelian categories, linear over a field $k$, with finite dimensional $\Hom$ spaces
and in which every object has finite length.
\begin{itemize}
 \item We say that $\wt {\cal C}$ is a mixed category if it is equipped with
an assignment of an integer to each isomorphism class of irreducibles 
$$w:\Irr\wt {\cal C}\rightarrow \Z$$ 
such that $\Ext^1(M,N)=0$ if $w(M) \leq w(N)$. We call $w(M)$ the weight of $M$.
\item A degree $d$ Tate-twist on a mixed category $\wt{\cal C}$ is an auto-equivalence $\langle d \rangle$ with the property 
that $w(M\langle d \rangle )=w(M)+d$. We will denote the $n$-fold iteration of $\langle d \rangle$ by $\langle nd \rangle$
\item Let $\wt {\cal C}$ be a mixed category with degree $d$ Tate-twist. Then a functor 
$v:\wt {\cal C} \rightarrow \cal C$
together with a natural isomorphism $v\cong v \circ \langle d \rangle$ is called degrading functor 
if it is exact, faithful and preserves semisimple objects.
\item A grading on $\cal C$ is a degrading $v:\wt {\cal C} \rightarrow \cal C$ such that every simple object of $\cal C$ 
lies in the essential image of $v$ and the natural map

\begin{equation}
 \bigoplus_n \Ext^i_{\wt {\cal C}}(M,N \langle nd\rangle ) \rightarrow \Ext^i_{\cal C}(vM,vN)
\end{equation}

is an isomorphism for all $M,N \in \wt {\cal C}$ and $i\in \N$.
\item A grading $v:\wt {\cal C} \rightarrow \cal C$ is called Koszul, if 
$$\Ext^i_{\wt {\cal C}}(M,N)=0$$
for $M,N\in \Irr\wt {\cal C}$ such that $w(M)\neq w(N)+i$.
\end{itemize}
\end{defi}
Let us explain, where the name grading comes from. Given a mixed category with degree $d$ Tate-twist we can massage it into a 
category enriched over $d\Z$-graded vector spaces by defining $\grHom_{nd}(M,N):=\wt {\cal C}(M\langle nd \rangle,N)$.

Similarly we define $\grExt^i_{nd}(M,N):=\Ext^i_{\wt {\cal C}}(M\langle nd \rangle,N)$. The grading condition 
can then be phrased as
\begin{equation}
v \text{ }\grExt^i(M,N)=\Ext^i(vM,vN) \text{ for all $i$}
\end{equation}
where the $v$ on the left hand side means taking the underlying vector space.

\begin{ex}
Typical examples of gradings can be obtained as follows:

Let $A$ be a finite dimensional non-positively graded algebra over a field $k$ with semisimple $A_0$.
Then any simple graded module is annihilated by $A_{<0}$ and hence 
concentrated in a single degree. We define its weight to be this degree.
This turns $A\-\grmod$ into a mixed category with degree $1$ Tate-twist by the rule 
$$M\langle 1 \rangle _k=M_{k-1}$$
The forgetful functor 
$$v:A\-\grmod\rightarrow A\- \mod$$
is a grading. It is Koszul, if and only if the grading on the ring $A$ is Koszul in the usual sense (after reversing the sign of the grading) \cite[2.1.3]{BGS}.
In fact any grading is equivalent to a profinite-dimensional variant of this example \cite[4.1.6]{BGS}. 

\end{ex}
Given an abelian category $\cal A$ we denote by $\Irr\cal A$ (resp. $\Proj \cal A$) the class of irreducible objects 
(resp. indecomposable projective objects) modulo isomorphism.
If $\wt{\cal C}\rar \cal C$ is a grading, $\Irr\wt{\cal C}$ (resp. $\Proj \wt{\cal C}$)
carries a $d\Z$-operation. We denote the quotient by $\Irr\wt{\cal C}/{d\Z}$ (resp. $\Proj \wt{\cal C}/d\Z$).

\begin{facts}\label{factsgradings}
 Let $v:\wt {\cal C} \rightarrow \cal C$ be a grading. Then $v$ induces bijections
$$\Irr\wt {\cal C} /d\Z \rightarrow \Irr\cal C \text{ and } \Proj \wt {\cal C} /d\Z \rightarrow \Proj  \cal C$$
\begin{proof}
 \cite[4.3.2]{BGS} .
\end{proof}
\end{facts}

\subsection{Graded multiplicities}\label{SubsecGradedMult}
Given a grading $\wt {\cal C} \rightarrow \cal C$ one can refine multiplicity formulas into equations between Laurent polynomials.
We fix some notation:
\begin{notation}
Let $V=\bigoplus V_n$ be a $\Z$-graded vector space, whose total dimension $\sum \dim  V_n$ is finite. Then we denote its graded dimension by 
$$\grdim  V:= \sum (\dim  V_n) \v^n \in \Z[\v,\v^{-1}]$$
We denote the graded dual of $V$ by $V^*$ and the involution of $\Z[\v,\v^{-1}]$ which exchanges $\v$ and $\v^{-1}$ by $(\_)^*$ as well.
Hence we obtain formulas $(\grdim  V)^*=\grdim  (V^*)$ and $\grdim  (V\otimes W)= \grdim  V \cdot \grdim  W$. 
\end{notation}
Now let $v:\wt {\cal C} \rightarrow \cal C$ be a grading with degree $d$ Tate-twist. Then the Grothendieck group $K(\wt {\cal C})$
is a graded $\Z[\v^d,\v^{-d}]$-module, where $\v^d\cdot [M]:=[M\langle d \rangle]$ and $\deg \v^d=d$.
By \ref{factsgradings} we know that $K(\wt {\cal C})$ is graded free:
$$K(\wt {\cal C}) \cong \Z[\v^d,\v^{-d}] \underset{\Z}{\otimes} K (\cal C) $$
For any $\Z[\v^d,\v^{-d}]$-basis $(L_i)$ of $K(\wt {\cal C})$ and an element $M\in K(\wt {\cal C})$, we define the multiplicity of 
$L_i$ in $M$ by the formula 
$$M=\sum [M:L_i] \cdot L_i$$
So multiplicities are Laurent polynomials $[M:L_i]\in \Z[\v^d,\v^{-d}]$.
Assume that $\wt {\cal C}$ is also of finite cohomological dimension. Then we have a pairing 
$K(\wt {\cal C}) \times K(\wt {\cal C}) \rightarrow \Z[\v^d,\v^{-d}]$
induced by
$$h([x],[y]):= \sum (-1)^i \grdim \grExt^i(x,y)  \text{ for } x,y\in \cal A$$
We note the following rules for computation with graded multiplicities:
\begin{lem}\label{lemGradedComputation}
 \begin{itemize}
\item We have $h(f M ,g N)= f^* g \cdot h(M,N)$ for $f,g \in \Z[\v^d,\v^{-d}]$.
\item Let $(P_i),(L_j)$ be two dual bases in the sense that $h(P_i,L_j) = \delta_{ij}$. Then we have 
$$[M:L_i]=h(P_i,M)$$
\item Let $(\Delta_i),(\nabla_j)$ be two dual bases in the sense that $h(\Delta_i,\nabla_j) = \delta_{ij}$. Then we have 
$$[M:\Delta_i]=h(M,\nabla_i)^*$$ 
\item Let $( \_ )^* :\cal A^{op} \rar \cal A$ be an equivalence such that $(M \langle d \rangle)^* \cong M^* \langle -d\rangle$
for all $M\in \cal A$. Then we have
$$[M:L_i]^*=[M^*:L_i^*]$$ 
\item Let $(P_i),(L_j)$ and $(\Delta_i),(\nabla_j)$ be two pairs of dual $\Z[\v^d,\v^{-d}]$-bases of $K(\wt {\cal C})$, in the sense that
$$h(P_i,L_j)=h(\Delta_i,\nabla_j)=\delta_{ij}$$
Then we have the following identities of Laurent polynomials
$$ [P_i:\Delta_j]=h(P_i,\nabla_j)^*=[\nabla_j:L_i]^*$$
\end{itemize}
\begin{proof}
The proof is straightforward.
\end{proof}
\end{lem}

\subsection{Gradings from geometry}
We will work in a more geometric situation from now on. Let $X_0$ be a stratified variety over $\F_q$. 
Following \cite{BGS}, we would like to turn the functor
$$\cal P_{\Lambda}(X_0,\K)\rar \cal P_{\Lambda}(X,\K)$$
into a grading.

\begin{ex}
Let $X_0=pt_0:=\spec \F_q$. Then $\cal P_{\Lambda}(pt_0,\K)$
is equivalent to the category of finite-dimensional $\K$-vector spaces equipped with an automorphism, 
which can be represented by a matrix with entries in $\O$.
The functor 
$$\cal P_{\Lambda}(pt_0,\K)\rar \cal P_{\Lambda}(pt,\K)$$
is the forgetful functor to vector spaces. Observe that we don't have a grading for two reasons:
\begin{itemize}
\item There are far too many possible eigenvalues.
\item Nontrivial Jordan normal forms give rise to extensions and hence $\grExt^1(\cal F_0,\cal G_0)$ is bigger then $\Ext^1(\cal F,\cal G)=0$.
\end{itemize}
In order to fix the first problem, we should only allow eigenvalues of the form $q^i$.
To fix the second problem, we should only allow semisimple automorphisms. 
\end{ex}
The generalization of these fixes to varieties $X=\bigsqcup X_\l$ is to 
only allow those mixed perverse sheaves, all of whose weight pieces are semisimple and Tate. 
This notion is however only well behaved if one imposes a purity condition on the variety:
\begin{defi}\label{condTate}
Let $X_0=\bigsqcup X_{\l,0}$ be a stratification. 
We say that it satisfies the $\BGS$-condition or that it is $\BGS$, if 
for all $i\in \Z$ and $\l, \mu \in \Lambda$ the sheaf $\cal \H^i(l_\mu^* {l_\l}_{!*}\cal \K [d_\l])$
vanishes if $i+d_\l$ is odd and is isomorphic to a direct sum of copies of $\K(\frac{-d_\l-i}{2})$ if $i+d_\l$ is even.
\end{defi}
We will often abuse language and say that $X=\bigsqcup X_\l$ is $\BGS$, leaving $X_0=\bigsqcup X_{\l,0}$
implicit.

\begin{ex}
By \cite[4.4.3]{BGS}  the full flag variety is $\BGS$. Using the
same techniques one can check that partial flag varieties are also $\BGS$.
\end{ex}

Let $X_0$ be an acyclically stratified variety satisfying the $\BGS$-condition. 
Denote by $\cal P^{\BGS}:=\cal P^{\BGS}_{\Lambda}(X_0,\K)$ the category of mixed perverse sheaves on $X_0$ with the property, 
that the graded pieces of the weight filtration $\Gr^W_i$ are semisimple and Tate. More precisely this means, that they are
isomorphic to a direct sum of $\IC_{\l,0}(\frac{-d_\l-i}{2})$ for even $d_\l+i$.
We will often denote objects of $\cal P^{\BGS}$ by $\cal F^{\BGS},\cal G^{\BGS},\ldots$ and their base-change to $X$ by $\cal F,\cal G,\ldots$.

\begin{thm}\label{thmBGSMAIN}
Let $X_0$ be an acyclically stratified variety satisfying the $\BGS$-condition. Then the functor
 $$\cal P^{\BGS}_{\Lambda}(X_0,\K)\rightarrow \cal P_{\Lambda}(X,\K)$$
is a Koszul grading with degree $-2$ Tate-twist given by $\langle -2 \rangle:= \K(1)\otimes \_$ the usual Tate-twist.
\begin{proof}
 \cite[4.4.4]{BGS}.
\end{proof}
\end{thm}
In particular the Grothendieck group $K(\cal P^{\BGS})$ becomes a graded 
free $\Z[\q,\q^{-1}]$-module, where we use the notation 
$$\q:=\v^{2}$$
In this geometric context one can see the grading on the $\Hom$-spaces directly in terms of Frobenius eigenvalues:
\begin{lem}\label{LemGradingFrobCorres}
 Let $\cal F_0,\cal G_0 \in \cal P^{\BGS}$. Then there is a canonical identification 
$$\grHom(\cal F_0,\cal G_0)=\Hom(\cal F,\cal G)$$
which identifies the degree $k$-space on the left hand side with the $q^{k/2}$ Frobenius-eigenspace on the right hand side.
In particular the Frobenius operation on $\Hom_{\cal P}(\cal F,\cal G)$ is semisimple and all eigenvalues are of the form $q^k$ for some $k\in \Z$.
\begin{proof}
Indeed using \cite[5.1.2.5]{BBD}, we compute:
\begin{align*}
\grHom(\cal F_0,\cal G_0)_{-2n} &= \Hom_{\P^{\BGS}}(\cal F_0 \langle -2n \rangle,\cal G_0) \\
&=\Hom_{\P_\Lambda(X_0,\K)}(\cal F_0 (n),\cal G_0) \\
&=(\Hom(\cal F,\cal G)(-n))^{\Fr} \\
&=\{f \in \Hom(\cal F,\cal G)| \Fr.f =q^{-n} \cdot f \}
\end{align*}
Here $(\_)^{\Fr}$ stands for taking Frobenius invariants.
Hence we have an injection $\grHom(\cal F_0,\cal G_0) \inj \Hom(\cal F,\cal G)$ which is a bijection
by dimension-coincidence.
\end{proof}
\end{lem}

\begin{rem}
 The category $\cal P^{\BGS}_{\Lambda}(X_0,\K)$ is closed under duality and contains incarnations of the usual suspects
which are unique up to isomorphism (after normalization to a suitable weight):
\begin{enumerate}
\item We have 
$$\IC^{\BGS}_\l:=\IC_{\l,0}\in \cal P^{\BGS}_{\Lambda}(X_0,\K)$$
\item There exists a unique lift of $P_\l\surj \IC_\l$ to a projective cover in $\cal P^{\BGS}_{\Lambda}(X_0,\K)$:
$$P_\l^{\BGS}\surj \IC_{\l}^{\BGS}$$ 
\item There exists a unique lift of $\Delta_\l\surj \IC_\l$ to a map 
$$\Delta^{\BGS}_\l\surj \IC_{\l}^{\BGS}$$
\item There exists a unique lift of $\IC_\l\inj \nabla_\l$ to a map 
$$\IC_\l^{\BGS}\inj \nabla_\l^{\BGS}$$
\end{enumerate}
\begin{proof}
\begin{enumerate}
\item True by definition.
\item True by \ref{factsgradings}.
\item Follows after passing to $\ol X_\l$ from the second point.
\item Dual to the third point.
\end{enumerate}
\end{proof}
\end{rem}
In fact we can describe $\Delta^{\BGS}_\l$ more naturally as $l_{\l!} \K[d_\l]$:
\begin{prop}\label{PropVermaInBGS}
 We have $\Delta_{\l,0}=\Delta^{\BGS}_\l$
\begin{proof}
 Since closed immersions commute with the weight filtration, we may assume that $X_\l$ is open.
Since a map $\cal F_0\rar \cal G_0$ is an isomorphism if and only if the induced map 
$\cal F \rar \cal G$ is one, we only need to show that the natural map
$$\Hom(\Delta_{\l,0},\Delta^{\BGS}_\l) \rightarrow \Hom(\Delta_\l,\Delta_\l)$$
is an isomorphism. This map is always injective and hence it suffices to check that both sides have the same dimension. 
The right hand side is one dimensional and for the left hand side we compute:
\begin{align*}
 \Hom(\Delta_{\l,0},\Delta^{\BGS}_\l)&=\Hom(\K[d_\l],j^! \Delta^{\BGS}_\l) \\
&=\Hom(\K[d_\l],\K[d_\l]) \\
&=\K
\end{align*}
\end{proof}
\end{prop}

\subsection{Koszulity of $\End(P^{\BGS})$}
Consider the category $\cal P^{\BGS}$. We want to rephrase \ref{thmBGSMAIN} into the statement, that 
$\cal P^{\BGS}$ is equivalent to the category of finitely generated modules
over a Koszul ring. Strictly speaking, this is wrong however. For example the latter category admits a degree $1$ Tate-twist, while 
the former only admits a degree $-2$ Tate-twist.

One can overcome this normalization problem by general nonsense. Indeed one can adjoin a square root of the Tate-twist to $\cal P^{\BGS}$, see \cite[below 4.1.4]{BGS}.
We denote the resulting grading by 
$$\wt \P:=(\P^{\BGS})^{\frac{1}{2}} \rar \P$$
 It is a Koszul grading with degree $1$ Tate-twist, in particular its Grothendieck group becomes a module over $\Z[\v,\v^{-1}]$.
The category $\wt \P$ has the bookkeeping advantage, that there exist lifts of important objects
``normalized to weight 0'':

\begin{lem}\label{LemKeyPos}
\begin{enumerate}
\item $\wt { \IC_{\l}}:=\IC_\l^{\BGS} \langle -d_\l \rangle$  is pure of weight $0$.
\item $\wt { P_{\l}}:=P_\l^{\BGS} \langle -d_\l \rangle$ is of weight $\leq 0$. Even better,
there is a short exact sequence 
$$0\rightarrow \ker \rightarrow \wt{P_\l}\rightarrow \wt{\IC_\l}\rightarrow 0$$
such that $\ker$ is of weight $<0$.
\item $\wt { \Delta_{\l}}:=\Delta_\l^{\BGS} \langle -d_\l \rangle$ is of weight $\leq 0$. Even better,
there is a short exact sequence 
$$0\rightarrow \ker \rightarrow \wt{\Delta_\l}\rightarrow \wt{\IC_\l}\rightarrow 0$$
such that $\ker$ is of weight $<0$.
\item $\wt { \nabla_{\l}}:=\nabla_\l^{\BGS} \langle -d_\l \rangle$ is of weight $\geq 0$. Even better,
there is a short exact sequence 
$$0\rightarrow \wt{\IC_\l} \rightarrow \wt{\nabla_\l}\rightarrow \cok \rightarrow 0$$
such that $\cok$ is of weight $>0$.
\end{enumerate}
\begin{proof}
 The first assertion is clear. The second is a consequence of \ref{LemProjWeightRad}. 
The third is a special case of the second, since $\Delta_\l$ becomes projective after passage to a smaller variety.
The fourth point is dual to the third.
\end{proof}
\end{lem}

\begin{lem}\label{LemProjWeightRad}
 The radical and the weight filtration on the projective cover $P_\l^{\BGS}\surj \IC_\l^{\BGS}$ coincide.
\begin{proof}
 This is a consequence of the explicit construction of $P_\l^{\BGS}$ in \cite[proof 4.4.8]{BGS}.
\end{proof}
\end{lem}

\begin{lem}\label{lemWeightsEndP}
  Let $X_0$ be an acyclically stratified variety, such that the $\BGS$-condition holds.
Then we have 
$$\grdim  \Hom(\wt P_\l, \wt P_\mu)=\sum_\nu [\wt \Delta_\nu: \wt{\IC}_\mu][\wt \Delta_\nu:\wt{\IC}_\l] $$
\begin{proof}
 This follows from \ref{lemGradedComputation}:
\begin{align*}
\grdim  \Hom(\wt P_\l, \wt P_\mu)&=[\wt P_\mu:\wt{\IC}_\l]\\
&= \sum_\nu [\wt P_\mu:\wt \Delta_\nu][\wt \Delta_\nu:\wt{\IC}_\l]\\
&=\sum_\nu [\wt \nabla_\nu: \wt{\IC}_\mu]^*[\wt \Delta_\nu:\wt{\IC}_\l] \\
&=\sum_\nu [\wt \Delta_\nu: \wt{\IC}_\mu][\wt \Delta_\nu:\wt{\IC}_\l]
\end{align*}
\end{proof}
\end{lem}

\begin{prop}\label{PropPBGSkoszul}
Let $X$ be an acyclically stratified variety satisfying the $\BGS$-condition.
Let $\wt P:= \bigoplus \wt P_\l$. Then $\grEnd(\wt P)$ is $\K$-Koszul. 
\begin{proof}
By \ref{thmBGSMAIN} and \cite[2.1.3]{BGS} we only need to show that
the grading on $\grEnd(\wt P)$ is non-positive and that the degree zero part is isomorphic to the ring $\bigoplus_{\l \in \Lambda} \K$.
We know from \ref{lemWeightsEndP} that 
$$\grdim  \Hom(\wt P_\l, \wt P_\mu)=\sum_\nu [\wt \Delta_\nu: \wt{\IC}_\mu][\wt \Delta_\nu:\wt{\IC}_\l] $$
But now \ref{LemKeyPos} gives us the following information:
\begin{itemize}
\item $[\wt\Delta_\nu:\wt{\IC}_\nu] \in 1+\v^{-1}\cdot \N[\v^{-1}]$
 \item $[\wt\Delta_\nu:\wt{\IC}_\mu] \in \v^{-1}\cdot \N[\v^{-1}]$ for $\nu \neq \mu$.
\end{itemize}
This finishes the proof.
\end{proof}
\end{prop}

\subsection{Lifts of projectives and extension of scalars}
We are now able to lift the projectives $P_\l^\O$ to $X_0$ such that $\K\otimes P^\O_{\l,0}=P^{\BGS}_\l$.
\begin{construction}\label{constrprojQlss}
Let $X_0=\bigsqcup X_{\l,0}$ be an acyclically stratified variety satisfying the $\BGS$-condition and $X_\l$ be a stratum.
Then the construction \ref{constrproj} of $P_\l^{\K} \in \cal P_{\Lambda}(X,\K)$ can be lifted to a construction of
$P_\l^{\BGS} \in \cal P^{\BGS}_{\Lambda}(X_0,\K)$.
\begin{proof}
We just tweak the construction \ref{constrproj}:
If $X_{\l,0}$ is open, we put $P^{\BGS}_\l:=\Delta_{\l,0}$. This is allowed by \ref{PropVermaInBGS}.

If $X_{\l,0}$ is not open, we construct $P^{\BGS}_\l$ again inductively by forming maximal extensions
$$0 \rightarrow E^*\otimes \Delta_{\mu,0} \rightarrow P^{\BGS}_\l \rightarrow P^{\BGS}_{Y,\l}\rightarrow 0$$
But this time we interpret $E=\grExt_{\cal P^{\BGS}}^1(P^{\BGS}_{Y,\l},\Delta_{\mu,0})$ 
as a graded vector space and $E^*\otimes M$ is defined to be a representative of the functor $E^*\otimes \grHom(\_,M)$ 
to graded vector spaces. More explicitly this means that the object $E^*\otimes M$ is a direct sum of shifted copies of 
$M$ and we fix once and for all an iso-transformation of functors valued in graded vectorspaces:
$$E^*\otimes \grHom(\_,M)=\grHom(\_,E^*\otimes M)$$
By definition we have $v P^{\BGS}_\l=P_\l$ hence it is the desired projective cover by \ref{factsgradings}.
\end{proof}
\end{construction}

In the case $\E=\K$ we just constructed some nice lifts to $X_0$. However we do not yet know about the cases 
$\E=\O,\F$. We show that lifts exist in all cases and identify the possible ambiguity.

\begin{construction}\label{constrproj0}
Let $X_0=\bigsqcup X_{\l,0}$ be an acyclically stratified variety satisfying $\IC^\O$-parity and $X_\l$ be a stratum. 
Then the construction \ref{constrproj} of $P_\l^{\E} \in \cal P_{\Lambda}(X,\E)$ can be lifted to a construction of
objects $P^\E_{\l,0} \in \cal P_{\Lambda}(X_0,\E)$.
\begin{proof}
We adjust again the construction \ref{constrproj}:
If $X_{\l,0}$ is open, we put $P_{\l,0}:=\Delta_{\l,0}$. 

If $X_{\l,0}$ is not open, we need to check that the extension class $e$ of 
$$0 \rightarrow E^*\otimes \Delta_\mu \rightarrow P_\l \rightarrow P_{Y,\l}\rightarrow 0$$
admits a preimage under the map 
$$\eta:\Ext^1_{X_0}(P_{Y,\l,0},E^*\otimes \Delta_{\mu,0}) \rightarrow \Ext^1_{X}(P_{Y,\l},E^*\otimes \Delta_\mu)$$
By \cite[(5.1.2.5)]{BBD}, we know that $\eta$ fits into a short exact sequence
$$\Hom_X(P_{Y,\l},E^*\otimes {\Delta_\mu})_{\Fr} \inj \Ext^1_{X_0}(P_{Y,\l,0},E^*\otimes \Delta_{\mu,0}) \surj \Ext^1_{X}(P_{Y,\l},E^*\otimes \Delta_\mu)^{\Fr} $$
Here $(\_)^{\Fr}$ (resp. $(\_)_{\Fr}$) stands for taking Frobenius invariants (resp. coinvariants).
But $e \leftrightarrow id_E$ is Frobenius invariant by construction, hence it can be lifted.
\end{proof}
\end{construction}
Note that the obstruction to uniqueness of $P_{\l,0}$ is $\Hom_X(P_{Y,\l},E^*\otimes {\Delta_\mu})_{\Fr}$.
While this space does not vanish in general, it is at least torsion:
\begin{lem}\label{keylemHomPED}
 Let $X_0$ be an acyclically stratified variety satisfying the $\BGS$-condition. Then 
$$\Hom(P^\K_{Y,\l},E^*\otimes {\Delta_\mu^\K})_{\Fr}=0$$
In particular there are no choices in our construction of $P^\K_{\l,0}$ and hence 
$$P_\l^{\BGS} \cong P^\K_{\l,0}$$
\begin{proof}
Let us write $\cal F$ instead of $\cal F^\K$ for the duration of this proof.
In the case of one stratum (i.e. $Y= \emptyset$) there is nothing to show.

Now assume that the assertions are true in the case of $n$ strata, let 
$X_0$ be stratified by $n+1$ strata, $U_0=X_{\mu,0}$ be an open stratum and $Y_0$ its closed complement.
We need to show that the weights on 
$$\Hom(P_{Y,\l},E^*\otimes {\Delta_\mu})$$
are all different from zero.

By our induction assumption, we already know that $P_{Y,\l,0} \cong P^{\BGS}_{Y,\l}$. Hence there is a short exact sequence by \ref{constrprojQlss}:
$$0 \rightarrow E^*\otimes \Delta_{\mu,0}\rightarrow P^{\BGS}_\l \rightarrow P_{Y,\l,0} \rightarrow 0$$ 

Now it is time to compare the characteristic polynomials $\chi_V=\chi_V(t)$ of Frobenius on various spaces $V$. 
We have that 
\begin{itemize}
\item $\chi_{\End(P_{Y,\l})}$ divides $\chi_{\Hom(P^{\BGS}_\l,P_{Y,\l})}$
\item $\chi_{\Hom(P_{Y,\l},E^*\otimes \Delta_\mu)}$ divides $\chi_{\Hom(P^{\BGS}_\l,E^*\otimes \Delta_\mu)}$
\item $\chi_{\Hom(P^{\BGS}_\l,P_{Y,\l})} \cdot \chi_{\Hom(P^{\BGS}_\l,E^*\otimes \Delta_\mu)} = \chi_{\End(P^{\BGS}_\l)}$ by projectivity.
\end{itemize}
By \ref{LemProjWeightRad}, we know that the top of $P^{\BGS}_\l$ is $\IC_{\l,0}$
and that there occur no other (untwisted!) copies of $\IC_{\l,0}$ in a composition series of $P^{\BGS}_\l$.
Since 
$$\grdim  \grHom(P^{\BGS}_\l,\cal F^{\BGS})=[\cal F^{\BGS}:\IC_{\l,0}]$$
it follows from \ref{LemGradingFrobCorres} that $1$ is an isolated zero of $\chi_{\End(P^{\BGS}_\l)}$.
On the other hand the Frobenius acts also with eigenvalue $1$ on the element $id \in \End(P_{Y,\l})$.
$$
\begin{xy}
 \xymatrix{ & \ar@{-}[ld] \chi_{\End(P^{\BGS}_\l)} \ar@{-}[rd] & \\
\chi_{\Hom(P^{\BGS}_\l,P_{Y,\l})} \ar@{-}[d]& & \chi_{\Hom(P^{\BGS}_\l,E^*\otimes \Delta_\mu)} \ar@{-}[d] \\
\chi_{\End(P_{Y,\l})} \ar@{-}[d] & & \chi_{\Hom(P_{Y,\l},E^*\otimes \Delta_\mu)} \\
t-1 & & 
}
\end{xy}
$$

Hence $1$ is not a zero of $\chi_{\Hom(P_{Y,\l},E^*\otimes \Delta_\mu)}$. In other words $\Fr-1$ acts 
bijectively on 
$\Hom(P_{Y,\l}^\K,E^*\otimes \Delta_\mu^\K)$ and hence $\Hom_X(P^\K_{Y,\l},E^*\otimes {\Delta_\mu^\K})_{\Fr}=0$.
\end{proof}
\end{lem}

\begin{cor}\label{corProjCoincide}
If the $\IC$-sheaves over $\O$ are parity, then we have 
$$\mathbb K \otimes P_{\l,0}^\O \cong P_\l^{\BGS}$$
\end{cor}

\section{Modular Koszul Duality}

\subsection{$\phi$-decomposability}
We have our Frobenius action on $\End(P^\O)$.
However this does not automatically lead to a grading by generalized eigenspaces, since $\O$ is not a field. 
In order to obtain a grading we need to make sure, that the eigenvalues are all distinct in $\F$.
To this end, we recall a notion from \cite{RSW}.
 \begin{defi}
Let $M$ be an $\O$-module and $\phi:M\rightarrow M$ be an $\O$-linear map. 
\begin{itemize}
\item We say that $M$ is $\phi$-decomposable,
if the inclusion
$$\bigoplus_{i \in \Z} M_i \inj M$$
is an isomorphism. Here we use the notation 
$$M_i:=\{m\in M | \exists n\in \N : (\phi-q^i)^n m=0\}$$
\item Let $I \subset \Z$. We say that $M$ has weights in $I$, if there exist natural numbers $n_i \in \N$
such that 
$$\prod_{i \in I} (\phi-q^i)^{n_i} =0$$ 
\end{itemize}
\end{defi}
Let us recall an instructive example from \cite{RSW}:
\begin{ex}\label{exNotPhiDec}
 Not every $\O$-module is $\phi$-decomposable, not even 
if all eigenvalues in sight are of the form $q^i$. For example suppose $q=l+1$ and let $M=\O^2$
be equipped with the automorphism
$$\left( \begin{array}{c c}
 1 & 1 \\
 0 & q
       \end{array}
\right)$$
It has the property, that $\K\otimes M$ is diagonalizable, while $\F\otimes M$ is not. 
Clearly this feature is incompatible with $\phi$-decomposability.
\end{ex}

\begin{lem}\label{lemWeightsEV}
 Let $M$ be a free $\O$-module with automorphism $\phi$. Then $M$ has weights in $I$, if and only 
if the eigenvalues of $\ol \Q_l \underset{\O}{\otimes} \phi$ are contained in the set $\{q^i| i \in I \}$.
\begin{proof}
 If the weights are contained in $I$, the minimal polynomial of $\ol \Q_l\otimes \phi$ has to divide $\prod_{i\in I} (t-q^i)^{n_i}$.
Hence all eigenvalues are contained in $\{q^i| i \in I \}$.

On the other hand, if all eigenvalues are contained in $\{q^i| i \in I \}$, the minimal polynomial of $\ol \Q_l\otimes \phi$ is of the form
 $\prod_{i\in I} (t-q^i)^{n_i}$ and we see $\prod_{i\in I} (\phi-q^i)^{n_i} =0$.
\end{proof}
\end{lem}

\begin{lem}\label{lemPhiDec}
Let $M$ be a free $\O$-module with automorphism $\phi$, whose weights lie in $I$. Assume that the elements $\{q^i| i \in I \}$ are all distinct in $\F$.
Then $M$ is $\phi$-decomposable.
\begin{proof}
 \cite[3.1.1]{RSW} 
\end{proof}
\end{lem}

\subsection{Separating Weights}
In this subsection we investigate which Frobenius eigenvalues occur in $\End(P^\O)$ for various spaces.
Let $\mathrm{Mon} \subset \Z[\q]$ be the set of monomials in the variable $\q$.

\begin{defi}\label{defWtX}
Let $V$ be a finite dimensional $\K$-vector space, along with an automorphism $F$. Assume further that 
all eigenvalues of $\ol \Q_l \otimes F$ are of the form $q^n$ for some $n\in \Z$ and $V$
is equipped with a $F$-stable decomposition: 
$$V=\bigoplus V_i$$
Then we define $ \mathrm{wt}(V)=\{\q^{n_1},\ldots , \q^{n_k}\}\subset \mathrm{Mon}$ to be the smallest subset, satisfying the following conditions:
\begin{itemize}
\item We have $1\in  \mathrm{wt}(V)$
\item For all $i$ there exists $k_i\in \Z$ such that the eigenvalues of $\ol \Q_l \otimes F$ restricted to $V_i$
are contained in $\{q^{n_1+k_i},\ldots , q^{n_k+k_i}\}$
\end{itemize}
If $X$ is an acyclically stratified variety satisfying the $\BGS$-condition, then
$$ \mathrm{wt}(X):= \mathrm{wt}(\End(P^\K))$$
where $\End(P^\K)$ is equipped with the Frobenius action (induced by the lift $P^{\BGS}$) and the decomposition
$$\End(P^\K)=\bigoplus \Hom(P^\K_\l,P^\K_\mu)$$
\end{defi}
Note that $[P_\l^{\BGS}:\IC_\mu^{\BGS}]$ records $ \mathrm{wt}(\Hom(P^\K_{\mu},P^\K_\l))$. Since in addition
$[P_\l^{\BGS}:\IC_\mu^{\BGS}]=\sum_\nu [\nabla_{\nu,0}:\IC_{\l,0}]^*[\Delta_{\nu,0}:\IC_{\mu,0}]$ 
one can compute $ \mathrm{wt}(G/P)$ in terms of Kazhdan-Lusztig polynomials.
\begin{defi}\label{defiWr}
Let $M=\{\q^{n_1},\ldots ,{\q^{n_k}} \}$ be a set of monomials, such that $n_1 < n_2 < \ldots < n_k$ and fix a prime $l$.
\begin{itemize}
\item We say that $M$ is separated (with respect to $l,q$), if $q^{n_1},\ldots , q^{n_k}$ are pairwise distinct elements of $\F$.
\item We define the weight range of $M$ to be 
$$ \mathrm{wr}(M):=n_k-n_1+1$$
For a variety $X$ we also use the notation $ \mathrm{wr}(X):= \mathrm{wr}( \mathrm{wt}(X))$.
\end{itemize}
\end{defi}
\begin{ex}\label{exbd}
Computations with Kazhdan-Lusztig polynomials yield the following:
\begin{itemize}
 \item If $X=\mathbb P^n$ is the projective space, say for $n>0$, then we have 
$$ \mathrm{wt}(X)=\{1,\q\} \text{ and }  \mathrm{wr}(X)=2$$
Observe that $ \mathrm{wt}(X)$ is independent of $n$.
\item If $X=\SL_{3}/B$ is the full flag variety of $SL_3$, then we have 
$$ \mathrm{wt}(X)=\{1,\q,\q^2,\q^3 \} \text{ and }  \mathrm{wr}(X)=4$$
\item More generally, if $X=G/B$ is any full flag variety, then we have
$$ \mathrm{wt}(X)=\{1,\q,\ldots, \q^{\dim  X} \} \text{ and }  \mathrm{wr}(X)=\dim  X+1$$
A complete proof can be found in the appendix \ref{propWeightsFullFlag}.
\item If $X=\Gr(k,n)$ is a Grassmannian, then we have 
$$ \mathrm{wt}(X)=\{1,\q,\ldots, \q^{\min(k,n-k)} \} \text{ and }  \mathrm{wr}(X)=\min(k,n-k)+1$$
A complete proof can be found in the appendix \ref{CorWeightsOnGrass}.

Observe again that $ \mathrm{wt}(X)$ stabilizes if we fix $k$, but let $n$ go to infinity! 
\end{itemize}
\end{ex}
By Dirichlet's theorem, for every $l> \mathrm{wr}(M)$ there exists a $p$ such that $ \mathrm{wt}(M)$ is separated.
Let us now formulate a criterion which guarantees Frobenius eigenvalues to give a grading:
\begin{prop}\label{propGradingFullFlag}
 Let $X$ be an acyclically stratified variety which satisfies the $\BGS$-condition and $\IC_\l^\O$-parity.
 If $ \mathrm{wt}(X)$ is separated, then $\Hom(P_\l^\O,P_\mu^\O)$ is $\Fr$-decomposable for all $\l,\mu$. 
\begin{proof}

By \ref{lemHomProjFree} we know that $\Hom(P_\l^\O,P_\mu^\O)$ is free.
By \ref{corProjCoincide} we know that $P^\O_{\l,0}\otimes \K=P_\l^{\BGS}$.
Hence the weights of 
$$\Hom(P_\l^\O,P_\mu^\O)$$
are recorded by $ \mathrm{wt}(X)$ and we can apply \ref{lemPhiDec}.
\end{proof}
\end{prop}

\subsection{Koszulity and formality theorems}

We are now able to prove our main result:

\begin{thm}\label{thmMainModKoszul}
 Let $X_0$ be an acyclically stratified variety. Suppose that the $\IC^\O$-sheaves are parity and that 
the $\\BGS$-condition holds. 
Let $P^\E=\bigoplus P_\l$. If $ \mathrm{wt}(X)$ is separated, then
$$A:=\End(P^\E)$$
admits an $\E$-Koszul grading.
\begin{proof}
We need to define a grading on $A=\End(P^\O)$. From \ref{propGradingFullFlag} we get a grading by Frobenius eigenvalues on each $\Hom(P_\l^\O,P_\mu^\O)$.
By \ref{corProjCoincide} we have $\K\otimes \Hom(P_\l^\O,P_\mu^\O)=\grHom(P_\l^{\BGS},P_\mu^{\BGS})$.
Since $\grEnd(P^{\BGS})$ needs to be normalized to become Koszul, we need to do the same thing with the grading on $\Hom(P_\l^\O,P_\mu^\O)$.
We shift it by $d_\mu-d_\l$ and hence get 
$$\K\otimes \Hom(P_\l^\O,P_\mu^\O)=\grHom(\wt P_\l,\wt P_\mu)$$
The grading on $\End(P^\O)$ is defined to be the direct sum of these gradings. Hence we have
$$\K\otimes \End(P^\O)=\grEnd(\wt P)$$
as graded rings.
The latter is $\K$-Koszul by \ref{PropPBGSkoszul}.
We now want to apply \ref{thmOKoszul} to $A=\End(P^\O)$, which means we have to check the following conditions:
\begin{enumerate}
\item We need to find orthogonal idempotents $e_\l$ for $\l\in \Lambda$, which exhibit $A_0$ as
$$ A_0=\prod_{\l\in\Lambda} \O $$

\item We need to check that 
$$\Ext^\bullet(L,L)$$
is a free $\O$-module, where $L:=\bigoplus L_\l$ and $L_\l=e_\l A /{ e_\l A}^{<0} $.
\end{enumerate}
So let us verify these two points.
\begin{enumerate}
 \item By \ref{PropPBGSkoszul} we know,
that the degree $0$ part of $\grEnd(\wt P)$ has a basis $e_\l$ consisting of the projections onto the $\l$-th projective. 
Let us show that the $e_\l$ are also an $\O$-basis of $A_0$.

Clearly they are linearly independent. 
Suppose that $\phi$ is a degree $0$ morphism, which does not lie in the span of the $e_\l$.
Since it does after tensoring with $\K$, there exist $a_\l \in \O$ and $n >0$ such that
$$l^n \cdot \phi = \sum a_\l e_\l $$
and say $a_\mu$ is not divisible by $l$. Multiplying the equation by $e_\mu$ 
shows that there exists a $\psi \in \End(P_\mu^\O)$ such that $l\cdot \psi = id_{P_\mu^\O}$.
This implies that the $\O$-module $\End(P_\mu^\O)$ is divisible. On the other hand it is also free.
A contradiction.

\item By \ref{LemExtGrpsTheSame}, we know that 
$$\Ext^\bullet(\IC^\O,\IC^\O)$$ 
is free. Hence we only need to check that $\IC_\l$ is mapped to $L_\l$ under \ref{thmMoritaProj}.

Consider the short exact sequence
$${\ker}_0 \inj P_{\l,0}\surj \IC_{\l,0}$$
Here the surjection is defined to be the composition $P_{\l,0}\surj \Delta_{\l,0} \surj \IC_{\l,0}$
and the injection is by definition its kernel.
Using our separatedness condition, the application of $\Hom(P,\_)$ yields a short exact sequence of graded modules
looking as follows:
$$\ker \inj e_\l A \surj \Hom(P,\IC_\l)$$
Since $\Hom(P,\IC_\l)$ is spanned (\ref{remGenerator}) by the map $P\rar P_\l\rar \Delta_\l \rar \IC_\l$, which exists over $\F_q$,
we see that $\Hom(P,\IC_\l)$ is concentrated in degree $0$. On the other hand we know by construction and \ref{LemKeyPos}
that $e_\l A$ has a degree $0$ part of rank $1$. It follows that $\ker={e_\l A}^{<0}$.
Hence $\IC_\l$ corresponds to $e_\l A/{e_\l A}^{<0}=L_\l$.
\end{enumerate}
\end{proof}
\end{thm}

\begin{rem}
Note that the method of proof also gives a very transparent comparison between modular and zero characteristic:
In the situation of \ref{thmMainModKoszul} let $A^\E=\E\otimes A$. Then the diagram 
$$
\begin{xy}
 \xymatrix{
 \P_\Lambda(X,\K) \ar[rr]^{\Hom(P^\K,\_)}  && \mod \-A^\K \\
 \P_\Lambda(X,\O) \ar[rr]^{\Hom(P^\O,\_)} \ar[d]_{\_ \otimes {\F}} \ar[u]^{\_ \otimes {\K}}  && \mod \-A^\O \ar[d]^{\_ \otimes {\F}} \ar[u]_{\_ \otimes {\K}} \\
 \D^b_\Lambda(X,\F) \ar[rr]^{\Hom(P^\F,\_)} && \D^b(\mod \-A^\F) \\
 }
\end{xy}
$$
commutes. Note the abuse of notation in the lower horizontal.
A priori $\Hom(P^\F,\_)$ is a functor $\D^b(\P_\Lambda(X,\F))\rar \D^b(\mod \-A^\F)$.
We composed it tacitly with the inverse of the realization functor.

In order to check that the squares commute, one exploits that the six functors commute with extension of scalars 
and that $P^\O \otimes \E = P^\E$.

\end{rem}


\section{Examples and applications}
If one wishes to apply \ref{thmMainModKoszul} in practice, one needs to calculate $ \mathrm{wt}(X)$ and 
decide whether the $\IC^\O$-sheaves are parity.
Calculating $ \mathrm{wt}(X)$, say in the case of flag varieties, can be done in terms of Kazhdan-Lusztig polynomials.
So this task is (in theory) straightforward though tedious.
Whether the $\IC^\O$-sheaves are parity is in general a very hard question.
\begin{rem}
Let us  remark, that $\IC^\O$ is always parity and $ \mathrm{wt}(X)$ is always separated for $l>>0$.
The reason is, that there is only $l$-torsion for finitely many $l$ in the stalks of the sheaf 
$\IC^\Z$ over $X_\C$. Hence in huge characteristics $\P_\Lambda(X,\O)$ is always Koszul.
\end{rem}

\subsection{Passage to complex numbers}
Let $X_\C$ be a partial flag variety over the complex numbers and $X_{\ol \F_q}$ the corresponding 
flag variety over $\ol \F_q$. As in \cite[7.1.4]{RSW} there is an equivalence 
of categories
$$\D^b_{(B)}(X_\C,\E)\cong \D^b_{(B)}(X_{\ol \F_q},\E)$$
By \cite[6.1.10]{BBD} it preserves standard objects, $\IC$-sheaves, the perverse $t$-structure, etc. 
Furthermore by Dirichlet's theorem, if $l> \mathrm{wr}(X)$ then there exists always a $p$ such that $ \mathrm{wt}(X_{\ol \F_p})$ is separated.
Hence we obtain the result \ref{MainThmIntro} announced in the introduction:
\begin{thm}\label{MainThmPartial}
Let $B\subset P \subset G$ be a complex semisimple group, along with a Borel and a parabolic subgroup. 
Let $X=G/P$ be the corresponding partial flag variety, stratified by $B$ orbits. 
Suppose that all $\IC^\O_w$ are parity. If $l> \mathrm{wr}(X)$
then there exists a Koszul ring $A$ and an equivalence of categories 
between perverse sheaves and finite dimensional modules over $A$:
$$\cal P_{(B)}(X,\F) \cong \mod \-A$$
\end{thm}
\subsection{Grassmannians}
In the case of Grassmannians \ref{MainThmPartial} applies nicely. The key point is that all $\IC^\O$-sheaves are parity since 
the relevant singularities admit small resolutions.
\begin{lem}
 Let $X$ be a cell stratified variety. Let $X_\l$ be a stratum and $\ol X_\l$ be its closure. 
Assume that there exists an even\footnote{A map between cell stratified varieties is even, 
if the preimage of each stratum is a union of strata and when restricted to strata, $\pi$ is a projection $\mathbb A^{n+k}\rar \mathbb A^k$ 
in suitable coordinates.} small resolution of singularities 
$$\pi:\wt X_\l\rar \ol X_\l$$
Then $\IC^\O_\l$ is parity.
\begin{proof}
We have $\IC^\O_\l= \pi_* \O[d_\l]$ by smallness and even morphisms preserve parity by \cite{JMWparity}.
\end{proof}
\end{lem}
Using the resolutions constructed in \cite{SmallResGrass} we obtain:
\begin{lem}
 If $X$ is a Grassmannian, then all $\IC^\O$-sheaves are parity.
\end{lem}
Hence we can apply \ref{MainThmPartial} to the category 
$\P_{(B)}(\Gr(k,n),\O)$
of perverse sheaves on a complex Grassmannian equipped with the usual stratification by Bruhat cells:
\begin{thm}
 Suppose that $l>\min(k,n-k)+1$ and let $P^{\O}$ be our usual projective generator of $\P_{(B)}(\Gr(k,n),\O)$.
Then $\End(P^\O)$ admits an $\O$-Koszul grading.
\end{thm}

\begin{rem}
In the case of characteristic zero coefficients perverse sheaves on Grassmannians were heavily investigated.
Braden gave an explicit quiver description of $\P_{(B)}(\Gr(k,n),\C)$ using microlocalization techniques in 
\cite{bradenPerverseOnGrass}. 
In the case $k=n-k$ Stroppel \cite{stroppelPerverseOnGrass} realized that Braden's algebra is isomorphic to $K^k_{n-k}$, the quasi hereditary cover 
of the Khovanov algebra. Later Brundan and Stroppel showed $\End(P^{\C})$ is isomorphic to $K^k_{n-k}$ using 
category $\cal O$ techniques. 
They studied this algebra thoroughly \cite{BrundanStroppel1},\cite{BrundanStroppel2},\cite{BrundanStroppel3},\cite{BrundanStroppel4} and obtained in particular 
that it is Koszul.
\end{rem}

\subsection{Flag varieties of small dimension}
In \cite{GeordieModInt} the $\IC$-sheaves on flag varieties of small dimension are examined. 
In particular it is shown there, that all $\IC^\O$-sheaves are parity for the following Dynkin diagrams:
\begin{itemize}
\item Type $A_n$ for $n \leq 6$ and any $l$.
\item Type $B_2$ and $l\neq 2$.
\item Type $D_4$ and $l\neq 2$.
\item Type $G_2$ and $l\neq 2,3$.
\end{itemize}
Hence \ref{MainThmPartial} can be applied in all of these cases.

\subsection{Category $\cal O$}
Let $G \supset B \supset T$ be a semisimple simply connected split algebraic group over $\F$, along with a Borel and a maximal torus.
Assume that $l$ is bigger than the Coxeter number.
Recall from \cite{SoergelPosChar} that we have the following objects in this setting:
\begin{itemize}
 \item A category $\cal O$. It is a certain subquotient of the category of rational $G$-representations.
The simple representations $L_w$ in $\cal O$ are parametrized by the Weyl group $W$.
They admit projective covers $P^{\cal O}_w$.
\item The algebra $C = \mathrm{Sym}(\mathfrak h)\otimes_{ \mathrm{Sym}(\mathfrak h)^W} \F$ of coinvariants. If 
we give it a grading by setting $\deg \mathfrak h=2$, then
it is canonically isomorphic 
to the cohomology ring of the complex flag variety of the Langlands dual to $G$.
$$\H^\bullet(G^L/B^L,\F)\cong C$$
\item A Soergel functor $\V:\cal O \rar C\-\mod$.
\item The hypercohomology functor $\mathbb H:\cal P_{(B^L)}(G^L/B^L,\F)\rar C\-\grmod$.
\item The subcategory $\cal C \subset C\-\grmod$ of graded Soergel modules. Its indecomposables up to shift are parametrized by $W$ and denoted by $D_w$.
By $|.|:\cal C \rar |\cal C|$ we denote the functor which forgets the grading.
\end{itemize}
The core of \cite{SoergelPosChar} is following:
\begin{thm}\label{thmSoergelPosCharMain}
In the above situation we have $\V(P_w^{\cal O})=|D_w|$ and $\mathbb H(\cal E_w)=D_w$, where $\cal E_w$ denotes the parity sheaf corresponding 
to the Bruhat cell $X_w$.
$$
\begin{xy}
\xymatrix{
\mathcal O \ar[rr]^{\mathbb V} && C\-\mod & \ar[l]  C\-\grmod && \ar[ll]_{\mathbb H} \cal P_{(B^L)}(G^L/B^L,\F) \\
 P_w^{\cal O} \ar@{|->}[rr] && |D_w| & \ar@{|->}[l] D_w && \ar@{|->}[ll] \cal E_w
 }
\end{xy}
$$
In addition $\V$ is fully faithful, when restricted to projectives and $\mathbb H$ is fully faithful on parity sheaves.
In particular for $P^{\cal O}:=\bigoplus P_w^{\cal O}$ and $\cal E:=\bigoplus \cal E_w$ we have:
$$\End(P^{\cal O}) \cong \Ext^\bullet(\cal E)$$
\begin{proof}
 \cite[2.6.1 and 4.2.1 and 2.8.2]{SoergelPosChar} 
\end{proof}
\end{thm}

\begin{cor}
Assume that $l>\mathrm{wr}(G^L/B^L)$ and that the $IC_w^\O \in \P_{(B^L)}(G^L/B^L,\O)$ are all parity.
Then there exist Koszul dual rings $A,A^!$ such that
$$\mod\-A \cong \P_{(B^L)}(G^L/B^L,\F) \text{ and } A^!\-\mod \cong \cal O$$
\begin{proof}
By \ref{thmMainModKoszul} we know, that $A:=\End(P^\F)$ admits a Koszul grading. Its Koszul dual ring is 
by definition $A^!=\Ext^\bullet(IC^\F)^{op}$, which coincides with $\End(P^{\cal O})^{op}$ by \ref{thmSoergelPosCharMain}.
\end{proof}
\end{cor}
A very similar statement was obtained in \cite[7.3.2.]{RSW}.

%
%
%

\appendix

\section{Grassmannians}
In this appendix we use 
the combinatorial framework of \cite{KLPolyGrass} to compute the set of weights $\mathrm{wt}(X)$ for a Grassmannian $X=\Gr(n,k)$.
Alternatively one could extract this information from \cite{BrundanStroppel2}.

The Schubert cells are in bijection with Young diagrams, which fit into a $k \times (n-k)$-rectangle.
In other words, they are in bijection with sequences of integers $\l=(\l_1,\ldots, \l_k)$ such that
$$n-k\geq \l_1 \geq \ldots \geq \l_k\geq 0$$
The bijection goes as follows:
Given $\l$ let $t_i:=\l_{k+1-i}+i$ and associate to it the cell
$$X_\l:=\{V\subset K^n| \dim V\cap (K^{t_i}\times 0)=i\}$$
Here $K$ is the field over which we consider the Grassmannian.
The dimension of $X_\l$ is given by the number of boxes of the corresponding Young diagram.

There is a combinatorial recipe in \cite{KLPolyGrass}, which allows to compute any graded multiplicity $[\wt \Delta_\mu : \wt{\IC}_\l]$ in the case of Grassmannians.
We will quickly review it. We start by recalling notation from \cite{KLPolyGrass}:

Let $\mathbf{P}:=\Z_{ >0}$ be the set of positive integers. 
A partition $\l=(\l_1,\ldots,\l_k)$ is a finite sequence of positive integers such that $\l_1 \geq \l_2 \geq \ldots \geq\l_k$.
In the case of Grassmannians sequences of the form $\l=(\l_1,\ldots,\l_k)$ with $\l_1 \geq \l_2 \geq \ldots \geq\l_k$, but where some of the $\l_i$ are zero
play a role. We will identify such a sequence with the partition obtained by deleting all zeros.

Given a partition $\l=\l_1 \geq \ldots \geq \l_k$ the corresponding Young diagram is defined to be the subset:
$$\{ (i,j) \in \mathbf{P}^2| i\leq \l_j\}$$
For example the partition $\lambda:=(3,2,2,1)$ corresponds to the diagram
$$
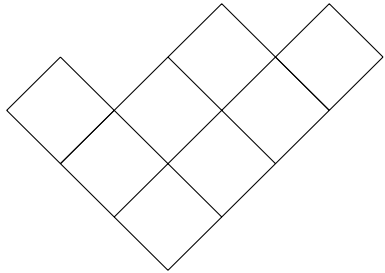
\begin{tikzpicture}[rotate=45]
 \draw (0,2) grid +(1,1);
 \draw (0,0) grid +(3,2);
 \draw (3,0) grid +(1,1);
\end{tikzpicture}
$$
From now on we will identify a partition with its Young diagram.
The set of Young diagrams is partially ordered by inclusion relation. A skew partition $\theta$ is a subset of $\mathbf{P}^2$,
which can be written as $\theta=\l -\mu$ where $\l\geq \mu$ are partitions. For example $(3,2,2,1)-(1,1,1)$ is a skew partition, which we visualize
as
$$
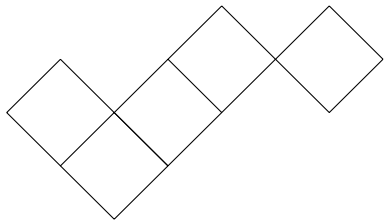
\begin{tikzpicture}[rotate=45]
 \draw (0,1) grid +(1,2);
 \draw (1,1) grid +(2,1);
 \draw (3,0) grid +(1,1);
\end{tikzpicture}
$$
We say that two skew partitions $\l-\mu$ and $\l'-\mu'$ are equivalent, if they are the same up to translation. More precisely
they are equivalent, if there exist $t,t' \in \N^2$ such that 
$$(\l-\mu)+t=(\l'-\mu')+t'$$
We say that a skew partition $\l-\mu$ is in normal form, if it is as close to the origin as possible.
More precisely we demand, that every equivalent skew partition is of the form $\l-\mu+t$ for some $t \in \N^2$.
Every skew partition is equivalent to a unique skew partition in normal form.

A border strip is a skew partition, which contains no $2\times 2$ square. 
A skew partition is connected, if for any two boxes there is a path inside the diagram which connects them and which does not 
pass through any corners of boxes. For example the above $(3,2,2,1)-(1,1,1)$ is not connected, while $(3,2,2,1)-(1,1)$ is connected.
We define the empty skew partition to be disconnected, so that any skew partition can be uniquely decomposed into its connected components.

Given $\mu \subset \l \supset \nu$ we say that $\l-\nu$ is a final segment in $\l-\mu$ if $\mu \subset \nu$. The outer border strip of $\l-\mu$ 
is defined to be the largest final segment, which is a border strip. The following picture indicates the outer border strip of $(4,3,2,2)-(2,2)$,
which happens to be disconnected:
$$
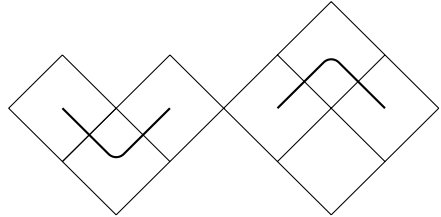
\begin{tikzpicture}[rotate=45]
 \draw (0,2) grid +(1,1);
 \draw (0,1) grid +(2,1);
 \draw (2,-1) grid +(2,2);
 \draw [thick,rounded corners] (0.5,2.5) -- (0.5,1.5) -- (1.5,1.5) ;
\draw [thick,rounded corners] (2.5,0.5) -- (3.5,0.5) -- (3.5,-0.5);
\end{tikzpicture}
$$
Let $\theta$ be a connected border strip, cbs for short. We say that $\theta$ is a Dyck cbs, if its rightmost and leftmost cell have the same level and no 
cell has strictly smaller level. Here the level of $(i,j)\in \mathbf{P}^2$ is by definition $i+j$.
For example the connected border strips
$$
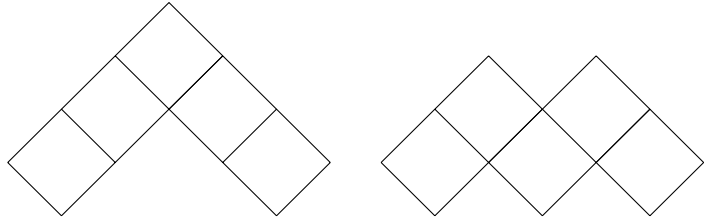
\begin{tabular}{l l}
$$
\begin{tikzpicture}[rotate=45]
 \draw (0,2) grid +(3,1);
\draw (2,0) grid +(1,2);
\end{tikzpicture}
$$
& 
$$
\begin{tikzpicture}[rotate=45]
 \draw (0,2) grid +(2,1);
\draw (1,1) grid +(2,1);
\draw (2,0) grid +(1,1);
\end{tikzpicture} 
$$
\\

\end{tabular}
$$
are both Dyck. On the other hand the following two cbs are not Dyck:
$$
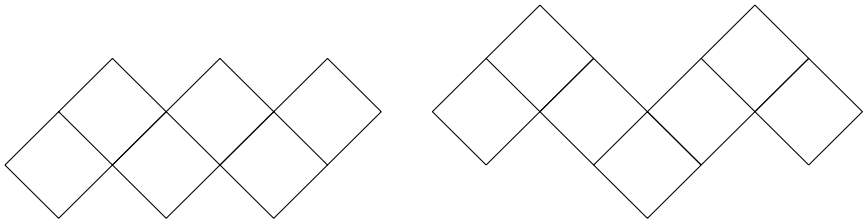
\begin{tabular}{l l}
$$
\begin{tikzpicture}[rotate=45]
\draw (0,2) grid +(2,1);
\draw (1,1) grid +(2,1);
\draw (2,0) grid +(2,1);
\end{tikzpicture}
$$
&
$$
\begin{tikzpicture}[rotate=45]
\draw (0,3) grid +(2,1);
\draw (1,1) grid +(1,2);
\draw (2,1) grid +(2,1);
\draw (3,0) grid +(1,1);
\end{tikzpicture}
$$
\\
\end{tabular}
$$
Let $n\in \N$. Then we define the set of Dyck skew partitions, to be those skew partitions which are obtained
by layering multiple Dyck cbs. More precisely being Dyck can be defined recursively:
\begin{defi}
The collection of Dyck skew partitions is the smallest $\N$-indexed system of sets of skew partition satisfying the following axioms:
\begin{enumerate}
\item The empty partition $\emptyset$ is Dyck of depth $0$.
\item Any Dyck cbs is a Dyck skew partition of depth $1$.
\item If $\eta$ is a skew partition with connected components $\eta^1,\ldots, \eta^k$ which are Dyck of depths $n_1,\ldots, n_k$ then 
$\eta$ is Dyck of depth $\sum n_i$.
\item If $\eta$ is connected, its outer border strip is Dyck of depth $a$ and $\eta'$ is Dyck of depth $a'$
then $\eta$ is Dyck of depth $a+a'$.
Here $\eta'$ denotes the skew partition which is obtained from $\eta$ by removing its outer border strip.

Given a Dyck skew partition $\eta$, we denote its depth by $\mathrm{dp}(\eta)$.
\end{enumerate}
\end{defi}
For example the diagram
$$
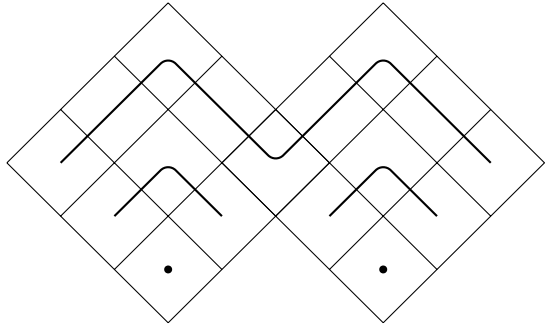
\begin{tikzpicture}[rotate=45]
 \draw (0,2) grid +(3,3);
 \draw (2,0) grid +(3,3);
\draw [thick,rounded corners] (0.5,4.5) -- (2.5,4.5) -- (2.5,2.5) -- (4.5,2.5) -- (4.5,0.5);
\draw [thick,rounded corners] (0.5,3.5) -- (1.5,3.5) -- (1.5,2.5);
\draw [thick,rounded corners] (2.5,1.5) -- (3.5,1.5) -- (3.5,0.5);
\fill (0.5,2.5) circle (1.5pt);
\fill (2.5,0.5) circle (1.5pt);
\end{tikzpicture}
$$
is Dyck of depth five. On the other hand the following skew partition is not Dyck. The problem arises in the third recursion:

$$
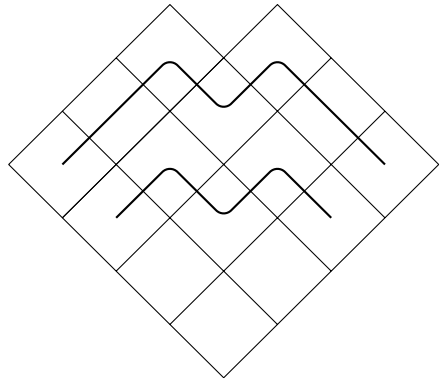
\begin{tikzpicture}[rotate=45]
 \draw (0,3) grid +(3,1);
 \draw (0,0) grid +(4,3);
\draw [thick,rounded corners] (0.5,3.5) -- (2.5,3.5) -- (2.5,2.5) -- (3.5,2.5) -- (3.5,0.5);
\draw [thick,rounded corners] (0.5,2.5) -- (1.5,2.5) -- (1.5,1.5) -- (2.5,1.5) -- (2.5,0.5);
\end{tikzpicture}
$$
Note that the Dyck property and depth are invariant under equivalence.

We now introduced enough terminology to state an equivalent of the main result of \cite{KLPolyGrass}:
\begin{thm}
 The graded multiplicity $[\wt \Delta_{\l}:\wt{\IC}_\mu]$ is zero if $\l -\mu $ is not Dyck. If $\l -\mu $ is Dyck we have
$$[\wt \Delta_{\l}:\wt{\IC}_\mu]= \mathrm{\bf v}^{\mathrm{dp} ( \l - \mu)}$$
\begin{proof}
 \cite[Thm 5.1.]{KLPolyGrass}
\end{proof}
\end{thm}
Hence we need to investigate the possible values of $\mathrm{dp}(\l - \mu)$.
\begin{prop}\label{propBoundGrass}
The following two sets coincide:
$$\{0,1,\ldots,  \min(k,n-k) \} = \{ \mathrm{dp}(\l-\mu) | \l-\mu \text{ is Dyck and $\l \subset k \times (n-k)$ }\} $$
\begin{proof}
Let $0\leq i \leq \min(k,n-k)$ and $\l:=(i,\ldots,i)$ be a square of $i^2$ many boxes and $\mu:=\emptyset$.
Then $\l=\l-\mu$ is Dyck of depth $i$.
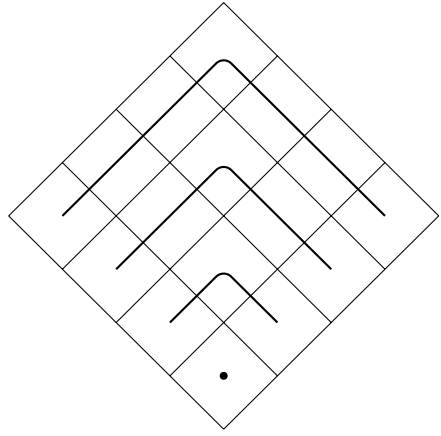
$$
\begin{tikzpicture}[rotate=45]
 \draw (0,0) grid +(4,4);
\draw [thick,rounded corners] (0.5,3.5) -- (3.5,3.5) -- (3.5,0.5);
\draw [thick,rounded corners] (0.5,2.5) -- (2.5,2.5) -- (2.5,0.5);
\draw [thick,rounded corners] (0.5,1.5) -- (1.5,1.5) -- (1.5,0.5);
\fill (0.5,0.5) circle (1.5pt);
\end{tikzpicture}
$$
 This shows that the left hand side is contained in the right hand side. 

By symmetry the other inclusion is implied by the following observation:
\begin{claim}
 $\mathrm{dp}(\l-\mu)\leq \l_1$ for all Dyck skew partitions $\l-\mu$.
\begin{proof}
We prove the claim by recursion on $\mathrm{dp}(\l-\mu)$:
\begin{itemize}
 \item For $\mathrm{dp}(\l-\mu)=0$ we have
$$\mathrm{dp}(\l)=0 \leq \l_1$$
\item If $\l-\mu$ is disconnected with components with normal form $\l^i-\mu^i$ we have 
$$\mathrm{dp}(\l-\mu)= \sum \mathrm{dp}(\l^i-\mu^i) \leq \sum \l^i_1 \leq \l_1$$
\item If $\eta=\l-\mu$ is connected and $\theta$ is the leftmost connected component of its outer border strip, we have
$$\mathrm{dp}(\eta)=\mathrm{dp}(\eta')+1 \leq (\l_1-1) +1 =\l_1-1+1=\l_1$$
Here we used the notation $\eta'$ for the skew partition $\eta$ with $\theta$ removed. 
\end{itemize}
\end{proof}
\end{claim}
\end{proof}
\end{prop}

\begin{cor}\label{CorWeightsOnGrass}
 Let $X=\Gr(n,k)$ be a Grassmannian equipped with the stratification into Bruhat cells. Then we have
$$\mathrm{wt}(X)=\{1,\q,\ldots, \q^{\min(k,n-k)} \} \text{ and } \mathrm{wr}(X)= \min(k,n-k)+1$$
\end{cor}

\section{The full flag variety}
In this appendix, we compute the set of weights for the full flag variety:
\begin{prop}\label{propWeightsFullFlag}
 Let $X=G/B$ then we have:
$$\mathrm{wt}(X)=\{1,\q,\ldots,\q^{\dim X} \}$$
\begin{proof}
In order to formulate the proof, we need some notation.
Let $P(\v)=\sum p_i \v^i$ and $Q(\v)=\sum q_i \v^i$ be two Laurent polynomials. Then we write $P \preceq Q$ or $Q \succeq P$ if 
$p_i \neq 0$ implies $q_i \neq 0$ for all $i\in \Z$. We write $P \sim Q$ if $Q\preceq P$ and $P \preceq Q$.
\begin{itemize}
 \item Let us show $\mathrm{wt}(X)\subset\{1,\q,\ldots,\q^{\dim X} \}$.
 By \cite[8.15]{HumphreysO} the Lowey length of a Verma module $\Delta_w$ is $l(w)+1$. Using 
the coincidence with the weight filtration \ref{LemProjWeightRad}, this implies 
$$[\wt\Delta_x:\wt{\IC}_y]  \preceq \sum_{i=0}^{l(x)} \v^{-i}$$
Hence using \ref{lemWeightsEndP} we compute:
\begin{align*}
 \grdim \, \grHom(\wt P_x, \wt P_y)&=\sum_w [\wt \Delta_w: \wt{\IC}_y][\wt \Delta_w:\wt{\IC}_x]\\
&\preceq \sum_w \sum_{i=0}^{l(w)} \sum_{j=0}^{l(w)} \v^{-i-j} \\
&\preceq \sum_{k=0}^{2 \dim X} \v^{-k}
\end{align*}
Now $\grHom(P_x,P_y)$ is concentrated in even degrees by \ref{LemGradingFrobCorres}.
This implies 
$$\mathrm{wt}(\Hom(P_x,P_y))\subseteq \{1,\q,\ldots ,\q^{\dim X} \} \text{ for all } x,y$$ 
Hence we have 
$$\mathrm{wt}(X) \subseteq \{1,\q,\ldots ,\q^{\dim X} \}$$
\item In order to establish the other inclusion, we only need to check that
$$\mathrm{wt}(\End(P_e)) \supseteq \{1,\q,\ldots ,\q^{\dim X} \}$$
By \cite[Lemma 2.1]{TiltingExercises}, the socle of $\Delta_w$ is  $\IC_e$ for all $w\in W$. Hence
\begin{align*}
\grdim \, \grEnd(P_e^{BGS}) &= \grdim \, \grEnd( \wt P_e)\\
&=\sum_w [\wt \Delta_w: \wt{\IC}_e][\wt \Delta_w:\wt{\IC}_e]\\
&\succeq \sum_w \v^{-2 l(w)} \\
&= 1+\q^{-1} \ldots +\q^{-\dim X}\\
\end{align*}
\end{itemize}
\end{proof}
\end{prop}

\normalfont

\bibliographystyle{alpha}
\bibliography{referenzen}

\end{document}